%% file: SCL_Bivirus.tex
\documentclass[5p,sort&compress]{elsarticle}

\usepackage{amsmath} 
\usepackage{amssymb}  

\usepackage{graphicx}          

\usepackage{tabulary}
\usepackage[english]{babel}
\usepackage{blindtext}
\usepackage[makeroom]{cancel}
\usepackage{amsfonts}
\usepackage{amsthm}  

\graphicspath{{../pdf/}{../jpeg/}}
\DeclareGraphicsExtensions{.pdf,.jpeg,.png}
\usepackage{pgf,tikz}
\usepackage{tabularx}
\usepackage{float}
\usepackage{subcaption}

\usepackage{enumitem} 
\usepackage{bbm} 
\usepackage{url}

\usepackage[strict]{changepage}
\usepackage{graphicx}
\usepackage{dcolumn}
\usepackage{bm}
\usepackage{hyperref}
\renewcommand{\eqref}[1]{Eq.~(\ref{#1})}  
\newcommand{\vect}[1]{\boldsymbol{#1}}

\DeclareMathOperator{\intr}{\mathrm{Int}}
\DeclareMathOperator{\diag}{\mathrm{diag}}

\newtheorem{theorem}{Theorem}
\newtheorem{lemma}{Lemma}
\newtheorem{proposition}{Proposition}
\newtheorem{assumption}{Assumption}
\newdefinition{remark}{Remark}

\definecolor{mygreen}{RGB}{72,160,66}
\definecolor{myblue}{RGB}{20,155,204}
\definecolor{myyellow}{RGB}{190,170,0}
\definecolor{myviolet}{RGB}{195,8,255}
\definecolor{myred}{RGB}{204,0,0}
\definecolor{myblue}{RGB}{0,66,255}

\begin{document}

\begin{frontmatter}

\title{Competitive epidemic networks with multiple survival-of-the-fittest outcomes\tnoteref{t1,t2}} 

\tnotetext[t1]{M. Ye is supported by the Western Australian Government through the Premier's Science Fellowship Program and the Defence Science Centre. The work of A.J., S.G, and K.H.J. is supported in part by the Knut and Alice Wallenberg Foundation, Swedish Research Council under Grant 2016-00861, and a Distinguished Professor Grant.}


\author[CUR]{Mengbin Ye\corref{cor1}}  
\author[ANU]{Brian D.O. Anderson}
\author[KTH]{Axel Janson}
\author[RIC]{Sebin Gracy}
\author[KTH]{Karl H. Johansson}

\cortext[cor1]{Corresponding author. Email: \texttt{mengbin.ye@curtin.edu.au}}

\address[CUR]{{Centre for Optimisation and Decision Science, Curtin University, Perth, Australia}}  
\address[ANU]{{School of Engineering, Australian National University, Canberra, Australia}}
\address[KTH]{{Division of Decision and Control Systems, School of Electrical Engineering  and  Computer  Science,  KTH Royal Institute of Technology, and Digital Futures, Stockholm, Sweden}}  
\address[RIC]{{Department of Electrical and Computer Engineering, Rice University, TX, USA}}


\begin{abstract}
We use a deterministic model to study two competing viruses spreading over a two-layer network in the Susceptible--Infected--Susceptible (SIS) framework, and address a central problem of identifying the winning virus in a ``survival-of-the-fittest'' battle. Existing sufficient conditions ensure that the same virus always wins regardless of initial states. For networks with an arbitrary but finite number of nodes, there exists a necessary and sufficient condition that guarantees local exponential stability of the two equilibria corresponding to each virus winning the battle, meaning that either of the viruses can win, depending on the initial states. However, establishing existence and finding examples of networks with more than three nodes that satisfy such a condition has remained unaddressed. In this paper, we address both issues for networks with an arbitrary but finite number of nodes. We do this by proving that given almost any network layer of one virus, there exists a network layer for the other virus such that the resulting two-layer network satisfies the aforementioned condition. 
To operationalize our findings, a four-step procedure is developed to reliably and consistently design one of the network layers, when given the other layer. Conclusions from numerical case studies, including a real-world mobility network that captures the commuting patterns for people between $107$ provinces in Italy, extend on the theoretical result and its consequences.
\end{abstract}

\end{frontmatter}

\section{Introduction}\label{sec:intro}

Mathematical models of epidemics have been studied extensively for over two centuries, providing insight into the process by which infectious diseases and viruses spread across human or other biological populations~\citep{hethcote2000mathematics,anderson1991_virusbook}. Models utilizing health compartments are classical, where each individual in a large population may be susceptible to the virus (S), infected with the virus and able to infect others (I), or removed with permanent immunity through recovery or death (R). Different diseases or viruses are modeled by including different compartments and specifying the possible transitions between the compartments. Two classical frameworks are Susceptible--Infected--Removed (SIR) and Susceptible--Infected--Susceptible (SIS), while further compartments can be added to reflect latent or incubation periods for the disease, or otherwise provide a more realistic description of the epidemic process. Moving beyond single populations, network models of meta-populations have also been widely studied, where each node in the network represents a large population and links between nodes represent the potential for the virus to spread between populations



Recently, increasing attention has been directed to network models of epidemics involving two or more viruses~\citep{wang2019coevolution}. Depending on the problem scenario, the viruses may be cooperative; being infected with one virus makes an individual more vulnerable to infection from another virus~\citep{newman2013interacting,cai2015avalanche,gracy2022modeling}. Alternatively, viruses may be competitive, whereby being infected with one virus can provide an individual with partial or complete protection from also being infected with another virus. 
For competitive models utilizing the SIS framework, a central question is whether each virus will persist over time or become extinct~\citep{sahneh2014competitive,liu2019bivirus,castillo1989epidemiological,wei2013competing,watkins2016optimal,santos2015bi,carlos2,yang2017bi,van2014domination,santos2015bivirus_conference,pare2021multi,janson2020networked,wang2012dynamics,ye2022bivirus_survey}. If a particular virus persists while others become extinct, it is said to have won the ``survival-of-the-fittest'' battle, and is also referred to as achieving a state of ``competitive exclusion''~\citep{carlos2,ackleh2005competitive}. An important problem is to identify the winning virus for the given initial states. It is also crucial to understand when multiple viruses may persist in the meta-population, resulting in a state of ``coexistence''. 

Our work considers a popular model for competing epidemics, namely two viruses in the SIS framework, described as a deterministic continuous-time dynamical system~\citep{carlos2,sahneh2014competitive,santos2015bi,liu2019bivirus,ye2022_bivirus}. The two viruses, termed virus~$1$ and virus~$2$, spread across a two-layer meta-population network; each layer represents the possibly distinct topologies for virus~$1$ and virus~$2$. In each population, individuals belong to one of three mutually exclusive compartments: infected by virus~$1$, or infected by virus~$2$, or not infected by either of the viruses. The competing nature implies that an individual infected by virus~$1$ cannot be infected by virus~$2$, and vice versa. An infected individual that recovers from either virus will do so with no immunity, and then becomes susceptible again to infection from either virus. 

Existing literature on bivirus networks has identified a variety of scenarios 
that specify the winning virus in the survival-of-the-fittest battle, regardless of the initial state~\citep{sahneh2014competitive,liu2019bivirus,santos2015bi,santos2015bivirus_conference,pare2021multi,janson2020networked}. In this paper, however, we address a key yet relatively unexplored question: \textit{are there networks such that either virus can prevail in the survival-of-the-fittest battle?} For a 3-node graph, with tree structure, \citep{carlos2} presented a necessary and sufficient condition for either of the viruses to prevail, depending on the initial states. However, the question has remained unanswered for networks with four or more nodes and general topology structure; the complexity arising from the coupled spreading dynamics of multiple nodes and two viruses makes it nontrivial to extend the approach in~\citep{carlos2}. 

The main contribution of this paper is to show that for any given finite number of nodes, and supposing that the network layer corresponding to one of the viruses has an arbitrary structure, one can systematically construct the network layer corresponding to the other virus such that either virus can survive in the resulting bivirus network, depending on the initial states. 
The main results are based on novel control-theoretic arguments, and begin by recalling that under a certain necessary and sufficient condition on the infection and recovery rates of the epidemic dynamics, the two equilibria associated with either virus winning the survival battle are both locally exponentially stable~\citep{ye2022_bivirus} (this condition extends the condition presented in~\citep{carlos2}). This ensures that there are initial states for which \textit{either} of the viruses can win the battle. While it is straightforward to check whether a given bivirus network satisfies the condition, the converse problems of \textit{existence} and \textit{design} of such networks are significantly more challenging. 
The reason why the demonstration of the existence of bivirus networks with more than three nodes satisfying the aforementioned condition has remained an elusive challenge is that the condition is expressed implicitly using complex nonlinear functions of the infection and recovery rates. Further, there have been no simple procedures to design or create the two network layers to satisfy the condition (a numerical example of which would resolve the existence question). 

We prove that, given almost any network layer of one virus, there always exists a network layer of the second virus such that the resulting bivirus network 
satisfies the aforementioned necessary and sufficient condition. 
We subsequently operationalize the theoretical results by developing a robust four-step procedure, starting with an essentially arbitrary network layer, to construct the other network layer to satisfy the condition. This allows one to generate bivirus networks that have two possible survival-of-the-fittest outcomes. Numerical examples involving small synthetic networks and a real-world large scale network are presented to demonstrate the procedure, and they show that the bivirus network model can exhibit a rich and complex set of dynamical phenomena, verifying the theoretical findings in \citep{carlos2}, including the presence, on occasions, of an unstable equilibrium where, in each population, both viruses coexist. Taken as a whole, our work offers insight into bivirus networks and the complex survival-of-the-fittest battles that unfold over them. \vspace{-6pt}
\subsection*{Paper Outline}
The paper unfolds as follows: We conclude this section by collecting most of the notations and certain preliminaries required in the sequel. The bivirus network model is detailed in Section~\ref{sec:model}, where we also formulate the problem of interest. The main results are presented in Section~\ref{sec:results}; the proofs have been relegated to the Appendix. Two case studies that illustrate the construction procedure and the resulting diverse limiting behaviour are provided in Section~\ref{sec:simulations}. Finally, conclusions and potential future work are given in Section~\ref{sec:conclusions}.

\subsection*{Notation and Preliminaries}
We use $I$ to denote the identity matrix, with dimension to be understood from the context. Let $A$ be a square matrix, with eigenvalues $\lambda_i$. We use $\rho(A) = \max_{i} \vert \lambda_i \vert$ and $\sigma(A) =  \max_i \mathfrak{Re}(\lambda_i)$ to denote the spectral radius and the spectral abscissa of $A$, respectively. If $\sigma(A) < 0$, we say $A$ is Hurwitz. The matrix $A$ is reducible if and only if there is a permutation matrix $P$ such that $P^{\top}AP$ is block upper triangular; otherwise $A$ is said to be irreducible.
For two vectors $x = \{x_i\}$ and $y = \{y_i\}$ of the same dimension, we write $x \leq y \Leftrightarrow x_i \leq y_i$ for all~$i$, and $x < y \Leftrightarrow x_i < y_i$ for all~$i$. We use $\vect 0_n$ and $\vect 1_n$ to denote the all-$0$ and all-$1$ column vectors of dimension $n$. We define the set \[\Delta = \{(x, y) \in \mathbb R^n_{\geq 0} \times \mathbb R^n_{\geq 0} : \vect 0_n \leq x + y \leq \vect 1_n \},\]
and its interior by $\intr(\Delta)$.

In this paper, we consider two-layer directed networks represented by the graph $\mathcal{G} = (\mathcal{V}, \mathcal{E}_A, \mathcal{E}_B)$, where $\mathcal{V} = \{1, \hdots, n\}$ is the set of nodes, and $\mathcal{E}_A \subseteq\mathcal{V} \times \mathcal{V}$ and $\mathcal{E}_B\subseteq\mathcal{V} \times \mathcal{V}$ are the ordered set of edges of the first and second layer, respectively. Associated with  $\mathcal{E}_A$ and $\mathcal{E}_B$ are the nonnegative adjacency matrices $A =\{a_{ij}\}$ and $B = \{b_{ij}\}$, respectively, that capture the edge weights. We define $a_{ij} > 0$ and $b_{ij} > 0$ if and only if $(j,i) \in \mathcal E_A$ and $(j,i) \in \mathcal E_B$, respectively, where $(j,i)$ is the directed edge from node $j$ to node $i$. A layer is strongly connected if and only if there is a path from any node $i$ to any other node $j$, which corresponds to the associated adjacency matrix being irreducible~\citep{berman1979nonnegative_matrices}. 

\section{Bivirus Network Model}\label{sec:model}

Following the convention in the literature~\citep{sahneh2014competitive}, we consider two viruses spreading over a two-layer network represented by the graph $\mathcal{G} = (\mathcal{V}, \mathcal{E}_A, \mathcal{E}_B)$, where $\mathcal{E}_A$ and $\mathcal{E}_B$ determine the spreading topology for virus~$1$ and virus~$2$, respectively. Each node represents a well-mixed population of individuals with a large and constant size; a well-mixed population means any two individuals in the population can interact with the same positive probability. Fig.~\ref{fig:epidemic_schematic} shows a schematic of the compartment transitions, and the two-layer network structure.

We define $x_i(t) \in [0, 1]$ and $y_i(t) \in [0, 1]$, $t\in \mathbb{R}_+$, as the fraction of individuals in population $i \in \mathcal{V}$ infected with virus~$1$ and virus~$2$, respectively.
In accordance with \citep{sahneh2014competitive,liu2019bivirus,santos2015bi}, the dynamics at node $i \in \mathcal{V}$ are given by
\begin{subequations}\label{eq:bivirus_node}
    \begin{align}
    \dot x_i(t) & = - x_i(t) + (1-x_i(t)-y_i(t))\sum_{j=1}^n a_{ij} x_j(t) \\
    \dot y_i(t) & = - y_i(t) + (1-x_i(t)-y_i(t))\sum_{j=1}^n b_{ij} y_j(t),
    \end{align}
\end{subequations}
with $a_{ij} \geq 0$ and $b_{ij} \geq 0$ being infection rate parameters. In fact, individuals in population~$j$ infected with virus~$1$ (resp. virus~$2$) can infect susceptible individuals in population~$i$ if and only if $(j,i) \in \mathcal{E}_A$ (resp. $\mathcal{E}_B$) at a rate $a_{ij}$ (resp. $b_{ij}$).  By defining $x(t) = [x_1(t), \hdots, x_n(t)]^\top$ and $y(t) = [y_1(t), \hdots, y_n(t)]^\top$, we obtain the following bivirus dynamics for the meta-population network:
\begin{subequations}\label{eq:bivirus_dynamics}
    \begin{align}
    &\dot x(t) = - x(t) + (I-X(t)-Y(t))Ax(t) \label{eq:virus1_dynamics} \\
    &\dot y(t) = - y(t) + (I-X(t)-Y(t))By(t), \label{eq:virus2_dynamics}
    \end{align}
\end{subequations}
where $X = \diag(x_1, \hdots, x_n)$, and $Y = \diag(y_1, \hdots, y_n)$. The system in \eqref{eq:bivirus_dynamics} has state variable $(x(t), y(t))$, and is in fact a mean-field approximation of a coupled Markov process that captures the SIS bivirus contagion process~\citep{sahneh2014competitive,santos2014bi,liu2019bivirus}. Note that we have taken the recovery rates for both viruses to be equal to unity for every population for the purposes of clarity. Importantly, this can actually be done without loss of generality when examining the stability properties of equilibria for the bivirus system (see \cite[Lemma~3.7]{ye2022_bivirus}). It is known from \cite[Lemma~8]{liu2019bivirus} that $\Delta$ is a positive invariant set for the bivirus dynamics in \eqref{eq:bivirus_dynamics}. 
Given that $x_i$ and $y_i$ represent the fraction of population~$i$ infected with virus~$1$ and virus~$2$, respectively, we naturally consider \eqref{eq:bivirus_dynamics} exclusively in $\Delta$, so that $x_i(t)$ and $y_i(t)$ retain their physical meaning in the context of the model for all $t\geq 0$.

We place the following standing assumption on the network topology.

\begin{assumption}\label{assm:strong_connect}
    The matrices $A$ and $B$ are irreducible, which is equivalent to both layers of $\mathcal{G}$ being, separately, strongly connected.
\end{assumption}

In the epidemiological context, this implies that there exists an infection pathway for the virus from any node to any other node. Strong connectivity is a standard assumption for \eqref{eq:bivirus_dynamics} (see~\citep{ye2022_bivirus,liu2019bivirus,pare2021multi}) and sometimes assumed without explicit statement~\citep{sahneh2014competitive} or is inherent from the problem formulation~\citep{carlos2}.


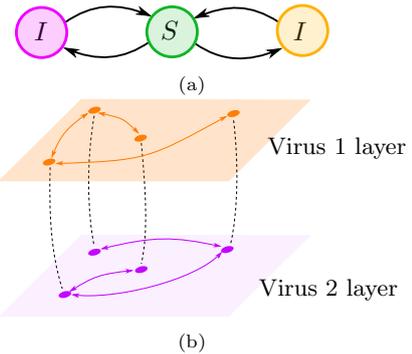
\begin{figure}
\centering
\subfloat[]{\def\svgwidth{0.55\linewidth}
	\input{transitions_bivirus_small.pdf_tex}\label{fig:transitions_bivirus}}
	\hfill
\subfloat[]{\def\svgwidth{0.6\linewidth}
	\input{two_layer_small.pdf_tex}\label{fig:two_layer}}
    \caption{Schematic of the compartment transitions and two-layer infection network. (a) Each individual exists in one of three health states: Susceptible ($S$), Infected with virus 1 ($I$, orange), or Infected with virus $2$, ($I$, purple). Arrows represent possible transitions between compartments. (b) The two-layer network through which the viruses can spread between populations (nodes). Note that the edge sets of the two layers do not need to match, so that virus $1$ can spread between two nodes but virus $2$ cannot, and vice versa. }
	\label{fig:epidemic_schematic}
\end{figure}

The second standing assumption of the paper is now stated, and as we will explain below, is in place due to our interest in survival-of-the-fittest outcomes.

\begin{assumption}\label{assm:rho}
    There holds $\rho(A) > 1$ and $\rho(B) > 1$.
\end{assumption}

There is always the healthy equilibrium $(x = \vect 0_n, y = \vect 0_n)$, where both viruses are extinct, and under Assumption~\ref{assm:rho}, it is an unstable equilibrium (in fact, a repeller such that all trajectories starting in its neighborhood move away from it)~\citep{liu2019bivirus}. With irreducible $A$ and $B$, there can be at most three \textit{types} of equilibria. There can be two ``survival-of-the-fittest'' equilibria $(\bar x, \vect 0_n)$ and $(\vect 0_n, \bar y)$, where $\vect 0_n < \bar x < \vect 1_n$ and $\vect 0_n < \bar y < \vect 1_n$~\citep{liu2019bivirus} (also referred to as boundary equilibria elsewhere in the literature~\cite{ye2022_bivirus}. Assumption~\ref{assm:rho} is relevant here: the necessary and sufficient conditions for $(\bar x, \vect 0_n)$ and $(\vect 0_n, \bar y)$ to exist are $\rho(A) > 1$ and $\rho(B) > 1$, respectively, see~\cite[Theorem~2 and Theorem~3]{liu2019bivirus}.
Moreover, $\bar x$ and $\bar y$ correspond to the unique endemic equilibrium of the classical SIS model considering only virus~$1$ and only virus~$2$, respectively~\citep{fall2007epidemiological,Lajmanovich1976,janson2020networked,liu2019bivirus}. These two separate single virus systems are given by 
\begin{subequations}\label{eq:single_SIS}
\begin{align}
     \dot x(t)& =-x(t)+(I-X(t))Ax(t),\label{eq:v1} \\
     \dot y(t)& =-y(t)+(I-Y(t))By(t).\label{eq:v2}
\end{align}
\end{subequations}
In the sequel, we provide a brief summary of the single virus system dynamics, including results useful for our theoretical analysis.
The third type of equilibrium involves \textit{coexistence} of both viruses, and any such equilibrium $(\tilde x, \tilde y)$ must necessarily satisfy $\tilde x > \vect 0_n$, $\tilde y > \vect 0_n$ and $\tilde x + \tilde y < \vect 1_n$~\citep{ye2022_bivirus}. Assumption~\ref{assm:rho} is the necessary condition for existence of a coexistence equilibrium, but it is not sufficient~\citep{ye2022_bivirus}. Bivirus systems can have a unique coexistence equilibrium which can be stable or unstable (uniqueness has only been established for $n\leq 3$)~\citep{carlos2,ye2022_bivirus}, or multiple (including an infinite number)~\citep{ye2022_bivirus,liu2019bivirus,anderson2022_bivirus_PH,doshi2022convergence}, or none~\citep{santos2015bi,ye2022_bivirus,carlos2}.

Next, we remark that recent results show that for a `generic' bivirus network\footnote{See \citep{ye2022_bivirus,ye2022bivirus_survey} for details on technical definitions of `generic' and `almost all').}, convergence to a stable equilibrium (of which there can be multiple) occurs for `almost all' initial conditions~\citep{ye2022_bivirus,ye2022bivirus_survey}. Hence, the key question is as follows: to which equilibrium does convergence occur? We are now ready to formulate the problem to be studied, which deals with bivirus networks with at least two \textit{stable} attractive equilibria.


{\bf Problem formulation.} In this paper, we study scenarios where either of the viruses can win the survival-of-the-fittest battle, under Assumptions~\ref{assm:strong_connect} and \ref{assm:rho}. Such scenarios are uncovered by examining the stability properties of the two equilibria $(\bar x, \vect 0_n)$ and $(\vect 0_n, \bar y)$ for the network dynamics in \eqref{eq:bivirus_dynamics}. Namely, we seek to study bivirus networks with conditions on $A$ and $B$ that ensure both $(\bar x, \vect 0_n)$ and $(\vect 0_n, \bar y)$ are locally exponentially stable. In such a scenario, there exist open sets $U \in \intr(\Delta)$ and $W \in \intr(\Delta)$, with non-zero Lebesgue measure and $U\cap W = \emptyset$, such that $\lim_{t\to\infty} (x(t), y(t)) = (\bar x, \vect 0_n)$ for all $(x(0), y(0)) \in U$ and  $\lim_{t\to\infty} (x(t), y(t)) = (\vect 0_n, \bar y)$ for all $(x(0), y(0)) \in W$. In context, the winner of a survival-of-the-fittest battle depends on the initial conditions.


The local exponential stability and instability of $(\bar x, \vect 0_n)$ and $(\vect 0_n, \bar y)$ can be characterized by analysis of the Jacobian of the right hand side of \eqref{eq:bivirus_dynamics}, evaluated at the two equilibria. Let $\bar X = \diag(\bar x_1, \hdots, \bar x_n)$ and $\bar Y = \diag(\bar y_1, \hdots, \bar y_n)$. We recall the following result of \cite[Theorem~3.10]{ye2022_bivirus}. 
\begin{proposition}\label{prop:boundary_stablility_cond}
Consider \eqref{eq:bivirus_dynamics} under Assumption~\ref{assm:strong_connect}. Then the following hold:
    \begin{enumerate}
    \item \label{ineq:1} The equilibrium $(\bar x, \vect 0_n)$ is locally exponentially stable if and only if $\rho((I-\bar X)B)<1$.
    \item \label{ineq:2}  The equilibrium $({\bf{0}}_n,\bar y)$ is locally exponentially stable if and only if $\rho((I-\bar Y)A)<1$.
\end{enumerate}
\end{proposition}

The equilibria $(\bar x, \vect 0_n)$ and $(\vect 0_n,\bar y)$ are unstable if the corresponding inequality in the proposition are reversed. Notice that these inequalities involve $\bar X$ and $\bar Y$ which are a nonlinear function of $A$ and $B$, respectively. In other words, $\rho((I-\bar X)B)$ depends on $B$ explicitly and $A$ implicitly, and hence the stability property of $(\bar x, \vect 0_n)$ is tied to the complex interplay between the $A$ and $B$ matrices, or equivalently, between the edge set and edge weights of the two layers. The same is true for $(\vect 0_n, \bar y)$. Thus, if one were provided $A$ and $B$, it is straightforward to check if the conditions hold, as there are iterative algorithms to compute $\bar x$ and $\bar y$, e.g.~\cite[Theorem~4.3]{mei2017epidemics_review} or \cite[Theorem~5]{vanMeighem2009_virus}. However, the inverse problems of existence and design are significantly more difficult to address. First, for an arbitrary number of nodes, proving the \textit{existence} of a bivirus system satisfying the above inequalities has remained an elusive challenge; firstly, is not automatically guaranteed that there exist $A$ and $B$ which satisfy one let alone both of conditions for local exponential stability given above. Second, no methods have been developed for designing bivirus networks with multiple survival-of-the-fittest outcomes. In the rest of this paper, we comprehensively address both of these issues.

The existing literature has identified sufficient conditions on $A$ and $B$ such that one survival-of-the-fittest equilibrium is locally stable and the other unstable, see ~\citep{santos2015bi,santos2015bivirus_conference,janson2020networked,sahneh2014competitive,liu2019bivirus,carlos2}. Note that in such a scenario, one can still have (multiple) locally stable coexistence equilibria~\citep{anderson2022_bivirus_PH}. Stronger sufficient conditions on $A$ and $B$ have been identified that ensure one of $(\bar x, \vect 0_n)$ and $(\vect 0_n, \bar y)$ is in fact globally stable in $\intr(\Delta)$ and the other unstable~\citep{santos2015bi,ye2022_bivirus}. Sufficient conditions on $A$ and $B$ can also be identified for both survival equilibria to be unstable~\citep{sahneh2014competitive,janson2020networked,carlos2}, and hence convergence must occur to a coexistence equilibrium (which may or may not be unique~\citep{anderson2022_bivirus_PH}).  For many of the conditions discussed, it is straightforward to i) prove there exist $A$ and $B$ that satisfy each of them, and ii) design one network layer, given the other layer, to meet these conditions. In contrast, existence and design of network topology to ensure both survival-of-the-fittest equilibria are locally stable is nontrivial, as we explained above, and has not been addressed in the literature to the best of our knowledge. 



\begin{remark}
    Our paper considers \eqref{eq:bivirus_dynamics} in the context of a meta-population model.
In some literature~\citep{sahneh2014competitive,liu2019bivirus}, node~$i$ is taken to be a single individual, and $x_i$ and $y_i$ are the probabilities that individual~$i$ is infected with virus~$1$ and virus~$2$, respectively. In other literature~\citep{carlos2}, the nodes may represent groups of individuals split according to some demographic characteristics, e.g. male or female. 
Depending on the modelling context, the diagonal entries of $A$ and $B$ may be zero (e.g. an individual cannot infect themselves), or $A$ and $B$ may be constrained to have the same zero and nonzero entry pattern (the two layers have the same topologies, but possibly different edge weights). 
Irrespective of the context, the dynamics are as given in \eqref{eq:bivirus_dynamics}, and the results in this paper are \textit{equally applicable} to various alternative physical/epidemiological interpretations of the model. This is because all of the aforementioned modelling frameworks are equivalent, see \citep{pare2018analysis}. 
\end{remark}

\section{Main Results}\label{sec:results}

The main theoretical results of this work are presented in two parts. 
First, we present an existence result which states that given almost any $A$ matrix, a corresponding $B$ matrix can be found to satisfy the desired stability condition. Second, we detail the four-step procedure for finding such a $B$ matrix, given a $A$ matrix. 

\subsection{Preliminaries on matrix theory and single virus SIS systems}

We recall relevant results from matrix theory and properties of the single SIS network model, required for deriving the main theoretical results. We say that a square matrix $A$ is a nonnegative (positive) matrix if all of its entries are nonnegative (positive). A nonnegative matrix $A$ is irreducible if and only if whenever $y=Ax$, with  $x\geq {\bf{0}}_n$, $y$ always has a nonzero entry in at least one position where $x$ has a zero entry. If $A$ is nonnegative and irreducible, then by the Perron--Frobenius Theorem~\citep{horn1994topics_matrix}, $\sigma(A) = \rho(A)$ is a simple eigenvalue, and we call it the Perron--Frobenius eigenvalue of $A$. The associated eigenvector can be chosen to have all positive entries, and up to a scaling, there is no other eigenvector with this property.  We say that $A$ is a Metzler matrix if all of its off-diagonal entries are nonnegative. By applying the Perron--Frobenius Theorem~\citep{horn1994topics_matrix} to a Metzler and irreducible $A$, similar conclusions on $\sigma(A)$ and the corresponding eigenvector can be drawn. A square matrix $A$ is an $M$-matrix if $-A$ is Metzler and all eigenvalues of $A$ have positive real parts except for any at the origin. If $A$ has eigenvalues with strictly positive real parts, we call it a nonsingular $M$-matrix, and a singular $M$-matrix otherwise~\citep{horn1994topics_matrix}.

Some properties of $M$-matrices and Metzler matrices, relevant to our theoretical results, are detailed as follows:
\begin{enumerate}
    \item For a (singular) $M$-matrix $F$, and any positive diagonal $D$, $DF$ is also a (singular) $M$-matrix.
    \item Let $F$ be an irreducible singular $M$-matrix. Then, for any nonnegative nonzero diagonal $D$, $F+D$ is an irreducible nonsingular $M$-matrix.
    \item Let $B$ be nonnegative irreducible and $D$ positive diagonal. Then for the Metzler matrix $-D+B$, there holds i) $\sigma(-D+B) > 0 \Leftrightarrow \rho(D^{-1}B) > 1$, ii) $\sigma(-D+B) = 0 \Leftrightarrow \rho(D^{-1}B) = 1$ and iii) $\sigma(-D+B) < 0 \Leftrightarrow \rho(D^{-1}B) < 1$.
    \item
For an irreducible nonnegative matrix $B$ and a  positive diagonal matrix $D$ with $d_{ii}<1$ $\forall i$, there holds $\rho(B)>\rho((I-D)B)$ and $\rho(B)>\rho(DB)$. 
\item
For a nonsingular irreducible $M$-matrix $F$, $F^{-1}$ is a positive matrix. 
\end{enumerate}
The first two results are easily proved from the property that all the principal minors of an $M$-matrix are positive in the nonsingular case and nonnegative in the singular case, see \cite[Theorem~4.31]{qu2009cooperative_book} and \cite[Chapter 6, Theorem 2.3 and Theorem 4.6]{berman1979nonnegative_matrices}. The third result 
is due to
\cite[Proposition~1]{liu2019bivirus}. The fourth is a consequence of \cite[Chapter 2, Corollary 1.5(b)]{berman1979nonnegative_matrices} and the irreducibility of both $(I-D)B$ and $DB$, which sum to $B$. The fifth is a consequence of \cite[Chapter 6, Theorem 2.7]{berman1979nonnegative_matrices}.

We now recall results for the single virus system in \eqref{eq:v1}, but obviously the same results will hold for \eqref{eq:v2}. The limiting behavior of \eqref{eq:v1} can be fully characterized by $\rho(A)$, see e.g.~\citep{Lajmanovich1976,mei2017epidemics_review,ye2021_PH_TAC}. Specifically, if $\rho(A) \leq 1$, then $\lim_{t\to\infty} x(t) = \vect 0_n$ for all $x(0) \in [0, 1]^n$. We call $\vect 0_n$ the healthy equilibrium. If $\rho(A) > 1$, then $\lim_{t\to\infty} x(t) = \bar x$ for all $x(0) \in [0, 1]^n \setminus \{\vect 0_n\}$, where $\vect 0_n < \bar x < \vect 1_n$ is the unique non-zero (endemic) equilibrium which is exponentially stable. Note that the equilibrium equation yields:
\begin{equation}\label{eq:equi_A}
    [-I + (I - \bar X) A]\bar x = \mathbf{0}_n,
\end{equation}
with $\bar X = \diag(\bar x)$. The following result characterises properties of the matrix on the left of \eqref{eq:equi_A}.

\begin{lemma}\label{lem:useful}
Consider the single virus system in \eqref{eq:v1}, and suppose that $\rho(A) > 1$ and $A$ is irreducible. With respect to \eqref{eq:equi_A}, the following hold:
\begin{enumerate}
    \item The matrix $-I + (I-\bar X)A$ is a singular irreducible Metzler matrix;
    \item $\sigma(-I + (I-\bar X)A) = 0$ is a simple eigenvalue with an associated unique (up to scaling) left eigenvector $u^\top$ and right eigenvector $\bar x$, with all entries positive, i.e. $u^\top \gg \vect 0_n$ and $\bar x \gg \vect 0_n$.
\end{enumerate}
\end{lemma}
\begin{proof}
Regarding the first statement, observe that $\vect 0_n \ll \bar x \ll \vect 1_n$ guarantees that $(I-\bar X)$ is a positive diagonal matrix, and hence $(I-\bar X)A$ is irreducible precisely when $A$ is irreducible. It is also nonnegative. Thus, $-I + (I-\bar X)A$ is an irreducible Metzler matrix. 
\eqref{eq:equi_A} implies that $\bar x$ is a null vector of the matrix $-I + (I-\bar X)A$, which accordingly is a singular matrix. Regarding the second statement, the conclusions immediately follow by viewing \eqref{eq:equi_A} in light of the properties of Metzler matrices detailed above.
\end{proof}

\subsection{Existence of two stable survival equilibria}\label{ssec:main_thm}

We now present the main theoretical result of this paper, showing that given almost any $A$ matrix, one can find a $B$ matrix (with $\rho(B)>1$) such that the two spectral radius inequalities in Proposition~\ref{prop:boundary_stablility_cond} are satisfied. 

Given $A$ with $\rho(A) > 1$, let $u^\top$ and $\bar x$ be the eigenvectors stated in Lemma~\ref{lem:useful}, normalised to satisfy $u^\top\bar x = 1$. Let $B'$ be any other nonnegative and irreducible matrix such that 
\begin{equation}
    [I - (I - \bar X) B']\bar x = \mathbf{0}_n, \label{eq:equi_B'}
\end{equation}
which similarly implies that $\bar x$ is a positive right eigenvector for $-I+(I - \bar X) B'$ associated to the simple eigenvalue at the origin. Let $v^\top$ be the associated left eigenvector, normalised to satisfy $v^\top \bar x = 1$. Notice that our definition of $B'$ implies that $A$ and $B'$ define the infection matrix for two separate single virus systems in \eqref{eq:v1} and \eqref{eq:v2} that have the same endemic equilibrium.

We require that $u$ and $v$ be linearly independent, and this can be straightforwardly achieved by selecting an appropriate $B'$ when given $A$. 
Indeed, we present Lemma~\ref{cor:main_paper} in the sequel, showing a procedure to select $B'$, when given $A$, to ensure the linear independence of $u$ and $v$. The main result follows, with proof in \ref{app:main_proof}.


\begin{theorem} \label{thm:doubly_stable_method}
Suppose that $A$ and $B'$ are irreducible nonnegative matrices, with $\rho(A) > 1$ and $\rho(B^\prime) > 1$, that satisfy \eqref{eq:equi_A} and \eqref{eq:equi_B'}, respectively. 
Suppose further that $u^\top$ and $v^\top$, as defined above, are linearly independent. Then there exists $\delta x \in \mathbb{R}^n$ with arbitrarily small Euclidean norm and satisfying 
\begin{align}
    u^{\top}[\bar X(I -\bar X)^{-1}] \delta x &> 0 \label{eq:deltax_1} \\
    v^{\top}[\bar X(I -\bar X)^{-1}] \delta x &< 0. \label{eq:deltax_2}
\end{align}
Furthermore, there exists $\delta B \in \mathbb{R}^{n\times n}$ such that $B^\prime + \delta B$ is an irreducible nonnegative matrix, and $\delta B$ also satisfies
\begin{equation}\label{eq:delta_B}
    \delta B \bar x=[(I- \bar X)^{-2}-B']\delta x.
\end{equation}
Then, with $B := B' + \delta B$, for the bivirus network in \eqref{eq:bivirus_dynamics}, both the survival-of-the-fittest equilibria $(\bar x, \vect 0_n)$ and $(\vect 0_n, \bar y)$ are locally exponentially stable, and $\bar y = \bar x + \delta x+o(\delta)$.
\end{theorem}

Provided $\delta x$ and $\delta B$ are sufficiently small, the resulting bivirus network in \eqref{eq:bivirus_dynamics} is such that either virus~$1$ or virus~$2$ may 
win a survival-of-the-fittest battle, depending on whether the initial states $(x(0), y(0))$ are in the region of attraction for $(\bar x, \vect 0_n)$ or $(\vect 0_n, \bar y)$, respectively. Our result does not exclude other limiting behavior, such as converging to a coexistence equilibrium where every population~$i$ has individuals infected with virus~$1$ and virus~$2$. This is because the regions of attraction for $(\bar x, \vect 0_n)$ and $(\vect 0_n, \bar y)$ together cannot cover all of $\intr(\Delta)$~\citep{chiang2015stability}, since there will be points in $\intr(\Delta)$ which are on the boundary of one or both regions of attraction (and thus cannot be part of the region). Indeed, there exist numerical examples where both $(\bar x, \vect 0_n)$ and $(\vect 0_n, \bar y)$ locally exponentially stable, and there are also multiple locally exponentially stable coexistence equilibria~\citep{anderson2022_bivirus_PH}.

\subsection{Systematic construction procedure}\label{ssec:procedure}

To begin, we provide a specific method for constructing a suitable $B'$, with proof given in \ref{app:main_proof}.
Let $e_i$ be the $i$-th basis vector, with $1$ in the $i$-th entry and $0$ elsewhere. 

\begin{lemma}\label{cor:main_paper}
    Let $A$ be an irreducible nonnegative matrix fulfilling \eqref{eq:equi_A} for some $\bar{x}$ such that $\mathbf{1}_n > \bar x > \mathbf{0}_n$. For a fixed but arbitrary $i \in \mathcal{V}$, let $z^\top \neq \vect 0_n$ be chosen to satisfy $z^\top \bar x = 0$ and the $j$-th entry $z_j < 0$ only if $a_{ij} > 0$. Then, there exists a sufficiently small $\epsilon$ such that $B' := A + \epsilon e_i z^\top$ is an irreducible nonnegative matrix. Moreover, $B'$ fulfills the conditions in the hypothesis of Theorem~\ref{thm:doubly_stable_method}: $\rho(B') > 1$, \eqref{eq:equi_B'} is satisfied, and $u$ and $v$ are linearly independent.
\end{lemma}



A procedure to systematically construct a bivirus network according to Theorem~\ref{thm:doubly_stable_method} is now presented.

{\bf Step 1.} Consistent with Theorem~\ref{thm:doubly_stable_method}, we begin by assuming that we are given an irreducible nonnegative matrix $A$ with spectral radius greater than 1.
Construct the matrix $B^\prime = A + \epsilon e_i z^\top$ according to Lemma~\ref{cor:main_paper}. 

{\bf Step 2.} With $u^\top$ and $v^\top$ as defined in Section~\ref{ssec:main_thm}, set $F = (I- \bar X)^{-2}-B'$ and $\tilde u^\top = u^{\top}\bar X(I -\bar X)^{-1}F^{-1}$ and $\tilde v^\top = v^{\top}\bar X(I -\bar X)^{-1}F^{-1}$. Note that $F$ is invertible and $F^{-1}$ is a positive matrix, as detailed in Appendix~\ref{app:main_proof}. Select two integers $j$ and $k$ for which $\tilde u_j/\tilde u_k > \tilde v_j/\tilde v_k$. This is possible since $u^\top$ and $v^\top$ (and thus $\tilde u^{\top}$ and $\tilde v^{\top}$ also) are linearly independent. Select $\alpha > 0$ to satisfy
\begin{equation*}
    \alpha \tilde u_j/\tilde u_k > 1 > \alpha \tilde v_j/\tilde v_k,
\end{equation*}
noting that such an $\alpha$ can always be found. Identify one positive entry in each of the $j$th row and $k$th row of $B'$, say $b_{jp}^\prime$ and $b_{kq}^\prime$. Set $\beta \in (0, b_{kq}^\prime \bar x_q)$. Finally, define the vector $s \in \mathbb R^n$ which has zeros in every entry except $s_k = -\beta$ and $s_j = \alpha\beta$. Compute $\delta x = F^{-1}s$.


{\bf Step 3.} To obtain $\delta B$, set all of its entries to be equal to zero, except that $\delta b_{kq} = -\beta/\bar x_q$ and $\delta b_{jp} = \alpha \beta/\bar x_p$. Then, set $B = B^\prime + \delta B$.

{\bf Step 4.} (If necessary). Check that the resulting $B$ satisfies the necessary and sufficient condition for local stability outlined in Proposition~\ref{prop:boundary_stablility_cond}, and if not, iterate Step~1--3 with different choices of $z^\top$, $\epsilon$, $\alpha$, and $\beta$. The theoretical analysis uses arguments centred on perturbation methods (see Appendix~\ref{app:main_proof}), and the $\delta x$ and $\delta B$ must be sufficiently small. The design choices of the construction method are $z^\top$, $\epsilon$, $\alpha$, and $\beta$, and hence one may need to adjust/tune these values in order to obtain a suitable $B$. 
Nonetheless, existence of such $B$ is guaranteed by Theorem~\ref{thm:doubly_stable_method}.

With $z^\top$, $\epsilon$, $\alpha$, and $\beta$ as the design choices, there are numerous potential options which give rise to different suitable $B$. For instance, since $\bar x > \vect 0_n$, $z$ must have at least one positive and one negative entry in order to satisfy $z^\top \bar x = 0$, and so one straightforward implementation is to set $z_k = 1$ and $z_j = -\bar x_k/\bar x_j$, for arbitrary $k,j$. For the selected index $j$, we need $\epsilon < \min_{\{l\in \mathcal{V}: a_{lj} > 0\}} a_{lj}/z_j$ to ensure that $B'$ is nonnegative.
Then, $B'$ is equal to $A$ except for the following entries: for the particular choice of $e_i$, $b'_{im} = a_{im} +\epsilon z_m$ for any $z_m\neq 0$. Next, $B$ is equal to $B'$ except the following entries: $b_{kq} = b'_{kq}-\beta/\bar x_q$ and $b'_{jp} = a_{jp}+\alpha\beta/\bar x_p$ for the indices $j,k,p,q$ identified in Step~2. 

If $A$ is a positive matrix, corresponding to an all-to-all connected virus~$1$ layer, then a more straightforward approach can be taken. We set $B'= A + \epsilon \vect 1_n z^\top$, with $z^\top \bar x = 0$ and $\epsilon$ sufficiently small to guarantee $B'$ is a positive matrix. Then, solve \eqref{eq:deltax_1} and \eqref{eq:deltax_2} for $\delta x$ using standard linear programming methods. Next, compute a solution $\delta B$ for \eqref{eq:delta_B} and apply a scaling constant to decrease the entries of $\delta B$ to ensure that $B = B'+\delta B$ remains a positive matrix. The challenge occurs when $A$ and $B'$ are not positive matrices, because any $\delta B$ satisfying \eqref{eq:delta_B} must have both positive and negative entries. This can be problematic if we obtain a solution $\delta B$ that has a negative entry where $B'$ has a zero entry but we also require $B$ to be nonnegative irreducible. The above four-step procedure resolves this issue, by producing a $\delta B$ whose single negative entry is in the same position corresponding to a positive entry in $B'$, and the former is smaller in magnitude than the latter.


In order to apply the four-step construction procedure, one requires knowledge of the infection matrix $A$, from which one can compute the endemic equilibrium $\bar x$ associated with the single virus system \eqref{eq:v1} (see below Proposition~\ref{prop:boundary_stablility_cond}). It is important to stress that only knowledge of the single virus system is needed, as opposed to knowledge of any bivirus system. From knowledge of $A$ and $\bar x$, one would construct a suitable $B^\prime$, and subsequently compute $\delta x$ and $\delta B$ as necessary. 

\begin{figure*} [!htb]
\centering
    \subfloat[]{\includegraphics[width= 0.3\textwidth]{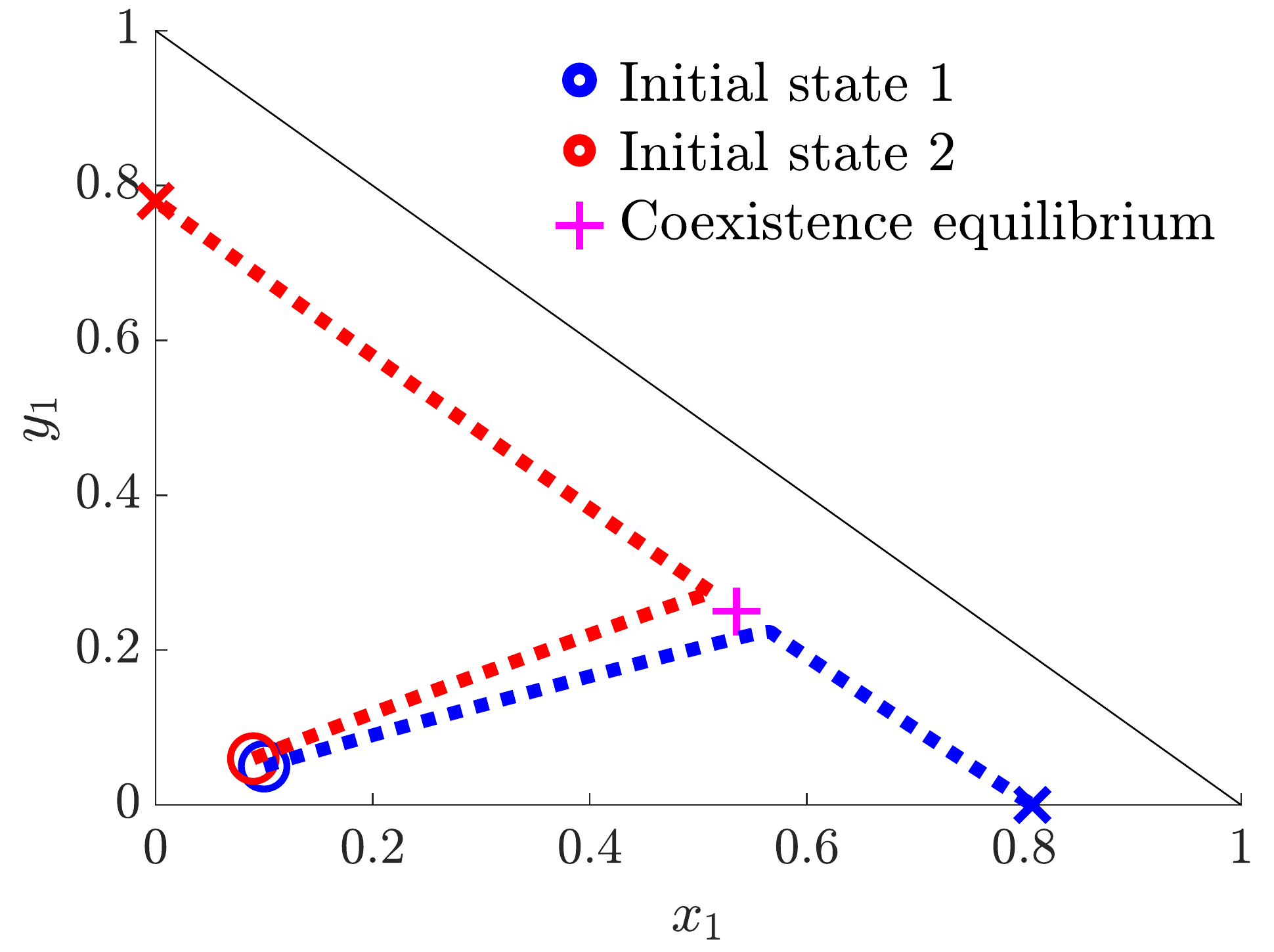}\label{fig:2a}}
    \hfill
    \subfloat[]{\includegraphics[width= 0.3\textwidth]{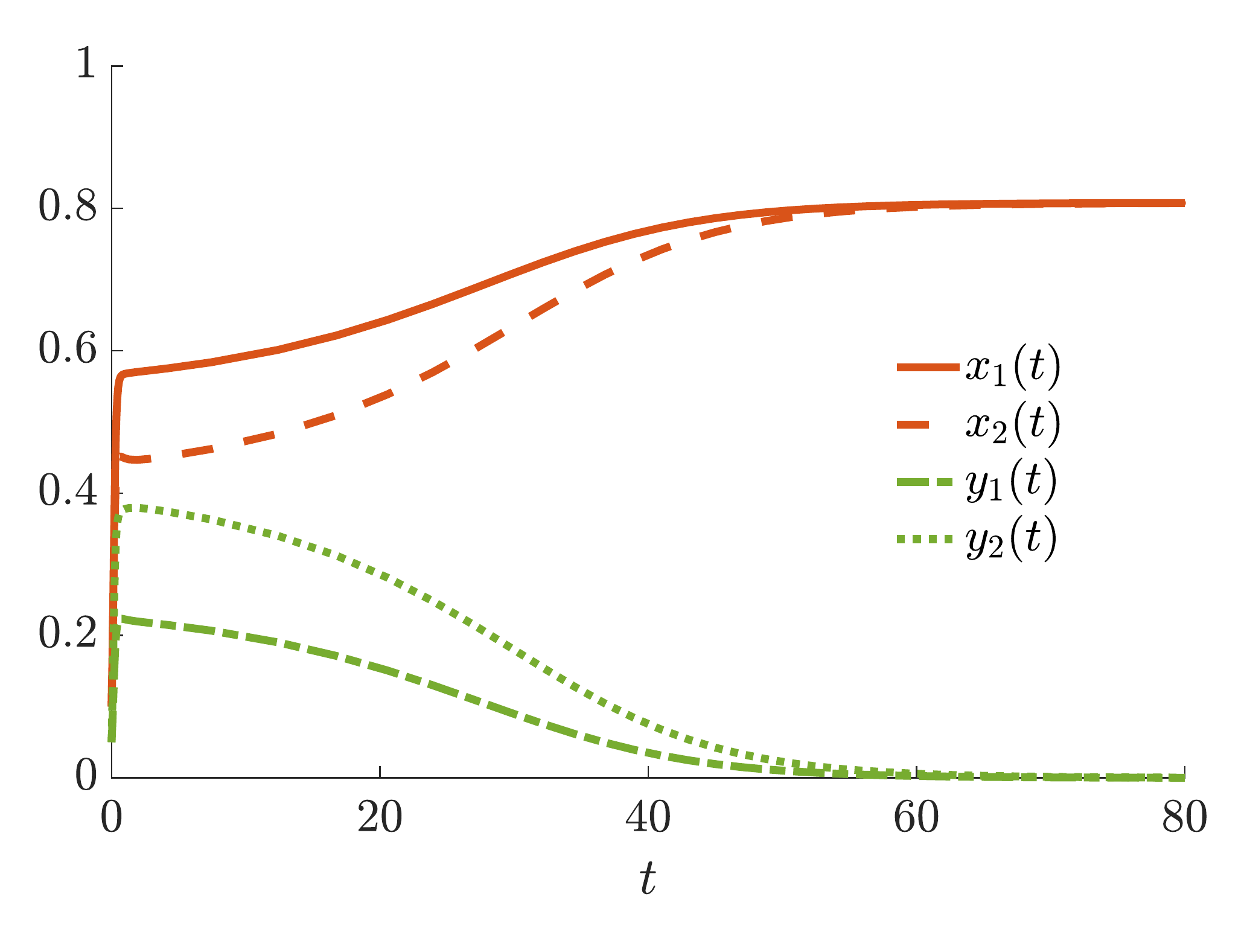}\label{fig:2b}}
    \hfill
    \subfloat[]{\includegraphics[width= 0.3\textwidth]{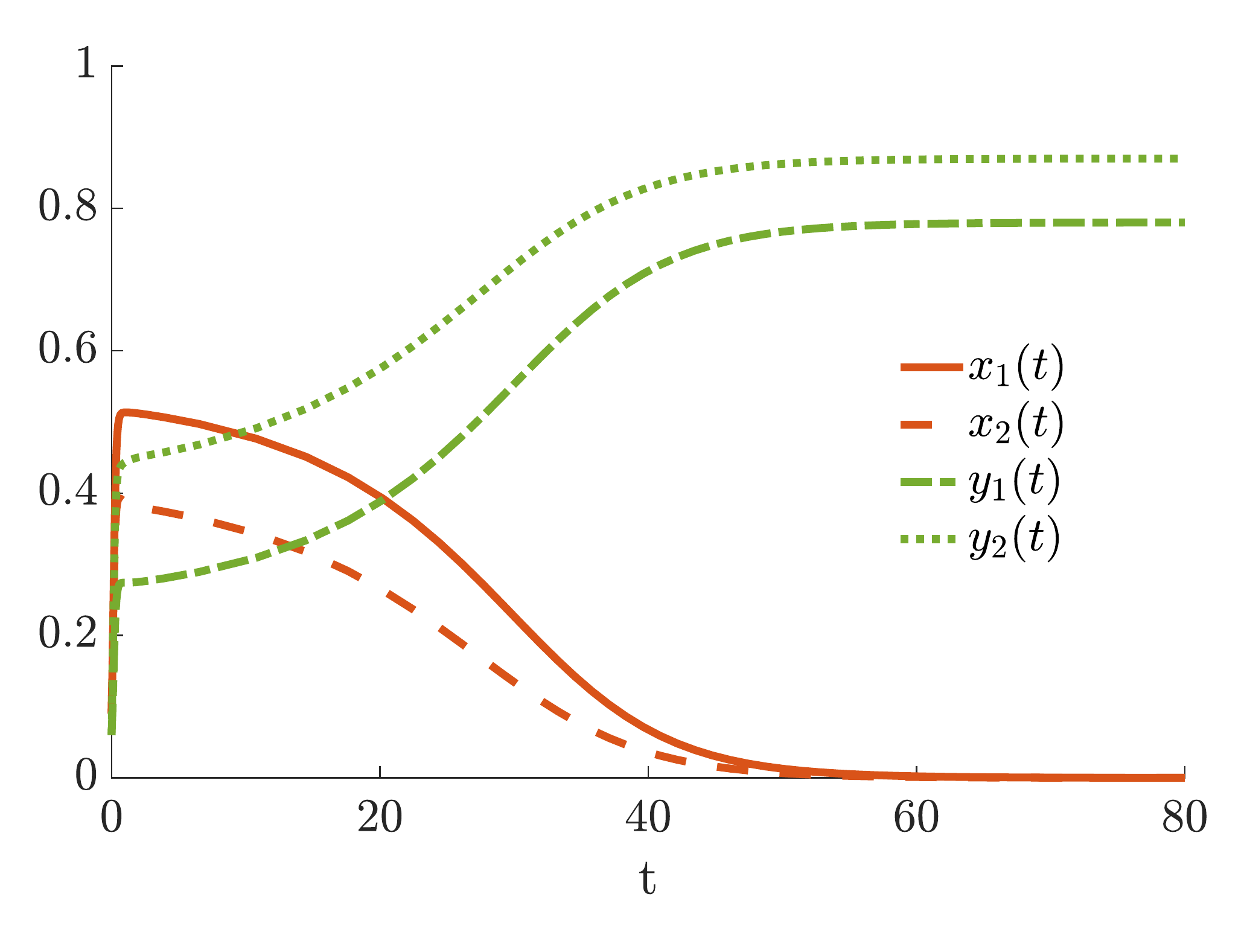}\label{fig:2c}}
    \hfill
    \\
    \subfloat[]{\includegraphics[width= 0.3\textwidth]{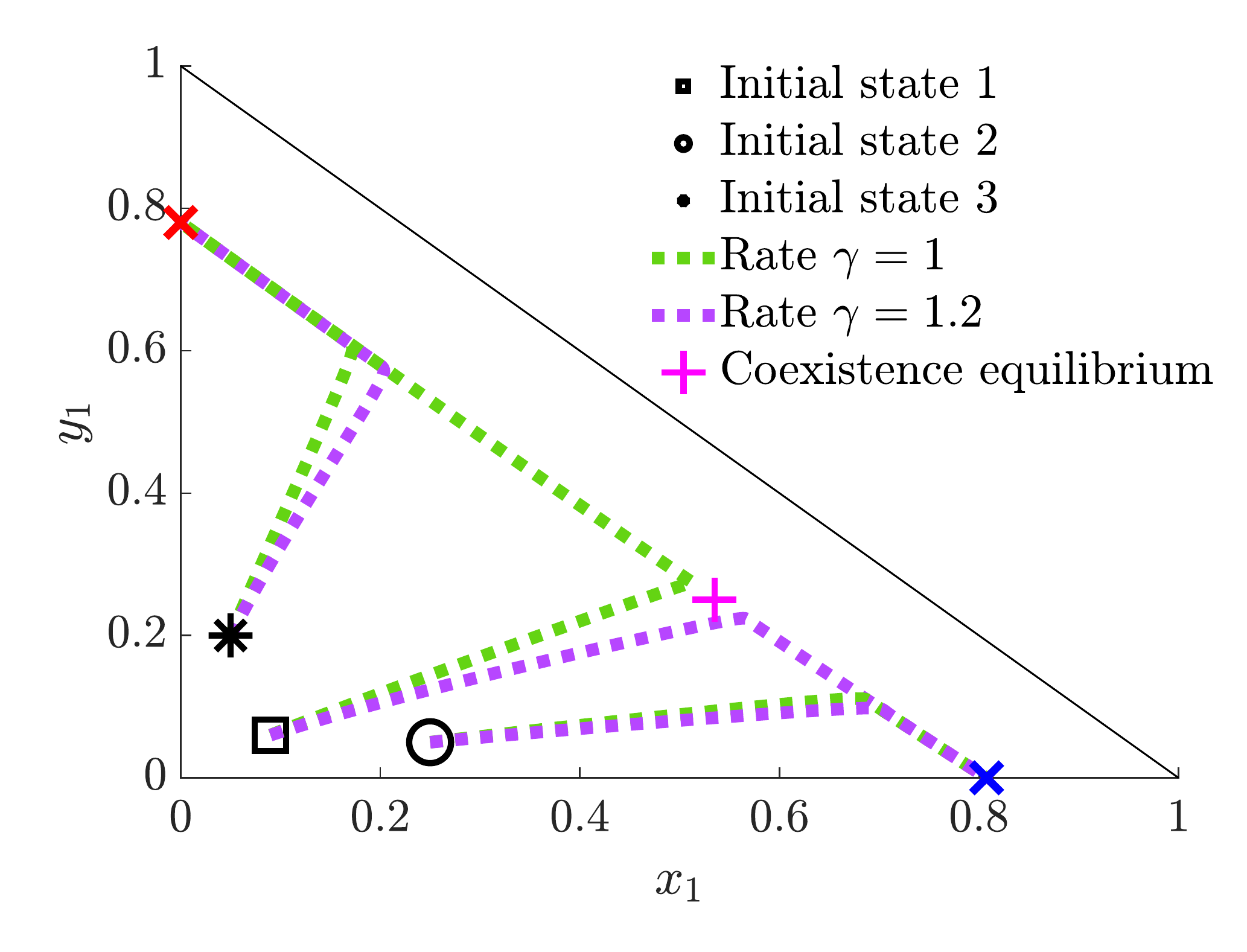}\label{fig:2d}}
    \hfill
    \subfloat[]{\includegraphics[width= 0.34\textwidth]{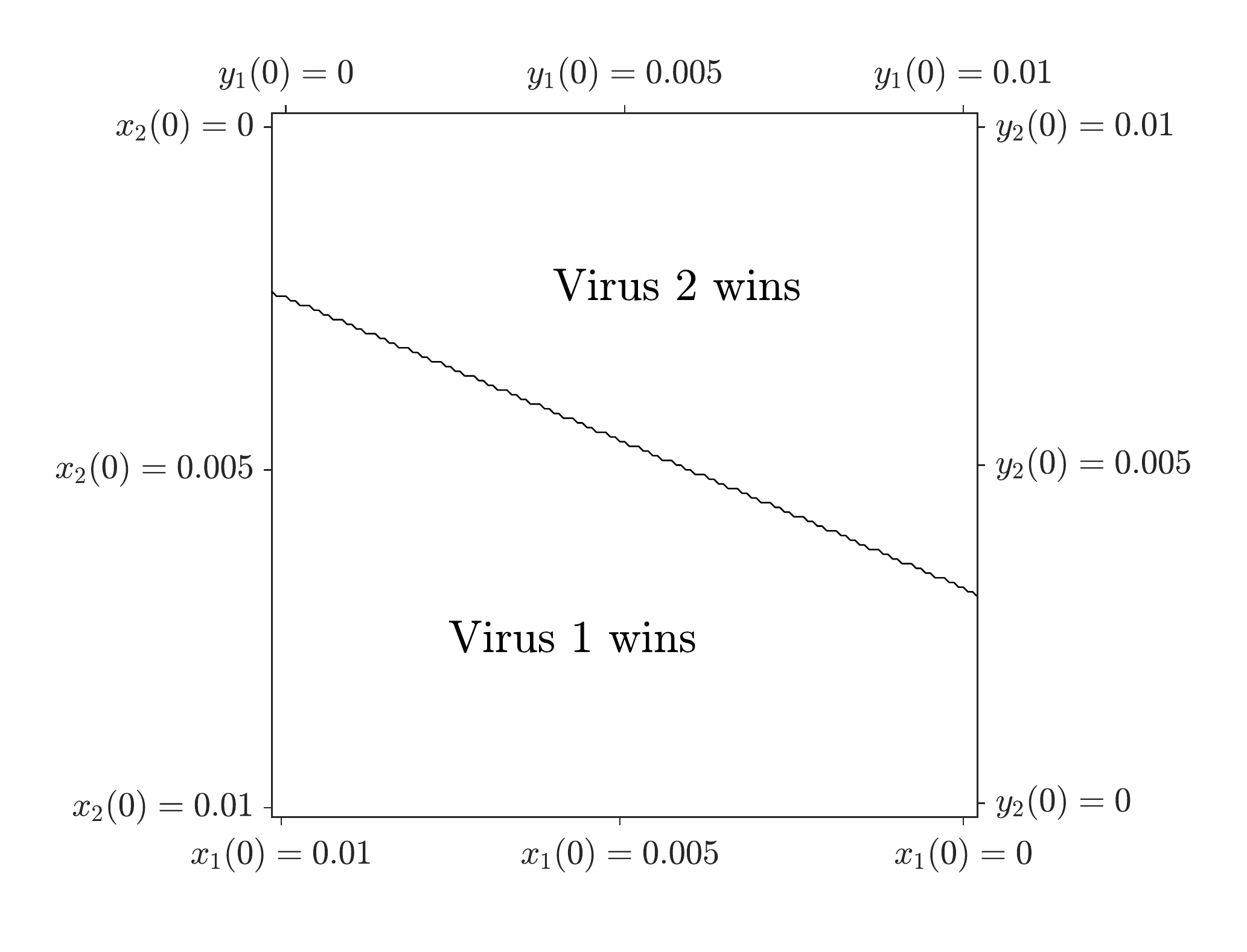}\label{fig:2e}}
    \hfill
    \subfloat[]{\includegraphics[width= 0.34\textwidth]{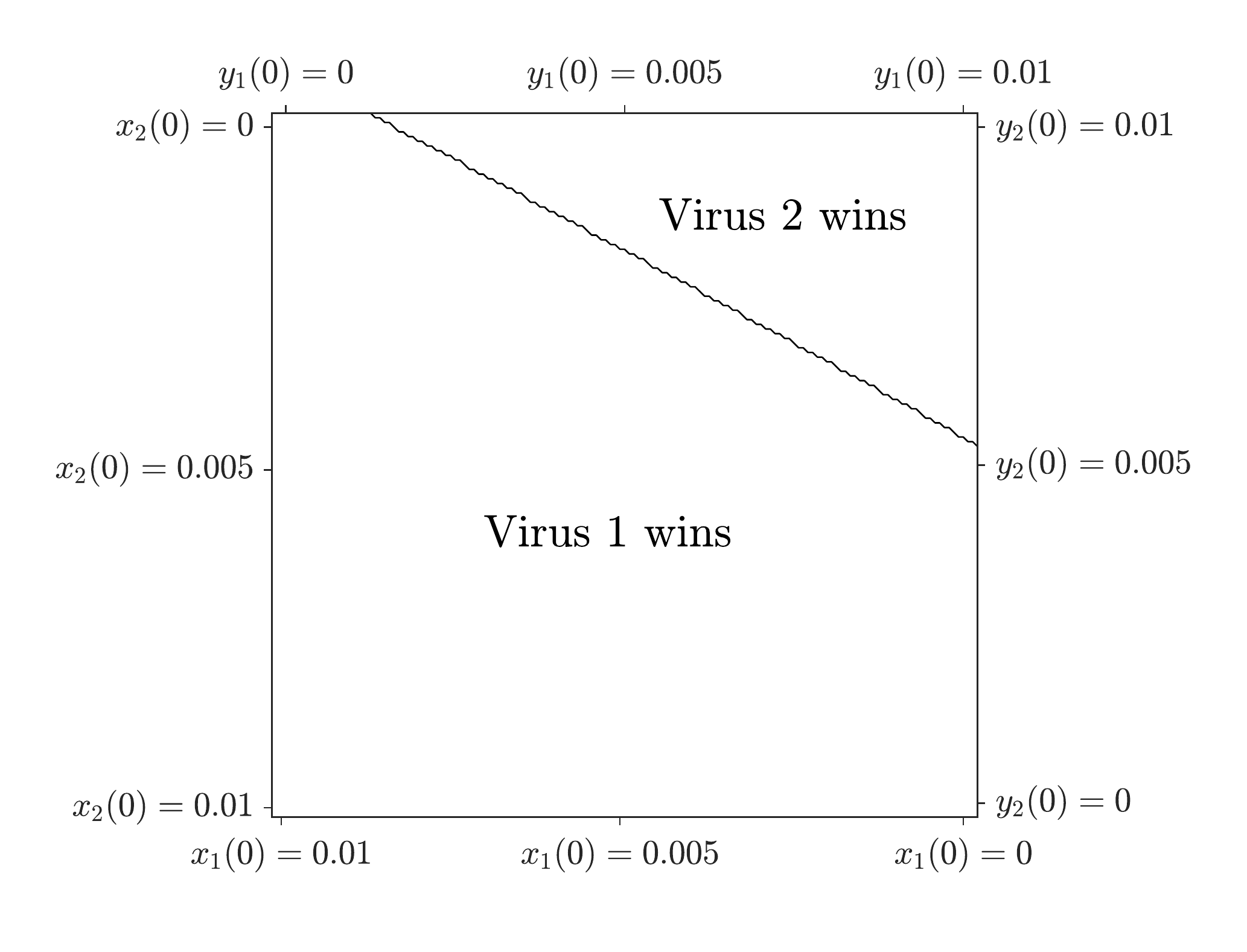}\label{fig:2f}}
    \caption{The dynamics of the two-node case study of \eqref{eq:bivirus_dynamics}. In (a), the trajectories $(x_1(t), y_1(t))$ are shown for two different initial states (blue and red); virus 1 and virus 2 win the survival-of-the-fittest battle in the blue and red trajectories, respectively. In (b) and (c), the time evolution of $(x(t), y(t))$ is shown for the blue and red initial states in (a), respectively. In (d), we show the trajectories $(x_1(t), y_1(t))$ for virus $1$ and virus $2$ of the same speed (green, $\gamma = 1$) and virus 1 that is $1.2$ times faster relative to virus 2 (purple, $\gamma = 1.2$), for different initial states. The winning virus for different initial states is recorded when (e) virus 1 and virus 2 are the same speed and (f) when virus 1 is faster than virus 2, with $\gamma = 1.2$. Note the line where the boundaries of the two regions meet forms part of the stable manifold of the unstable coexistence equilibrium.  }    \label{fig:2}
\end{figure*}


\section{Simulation Case Studies}\label{sec:simulations}

We now present two case studies to illustrate the procedure and the diverse limiting behavior that can be observed, including different survival-of-the-fittest outcomes. Full code at \url{https://github.com/lepamacka/bivirus_code} and \url{https://github.com/mengbin-ye/bivirus}. 

\subsection{Two-node case study}
We consider a setting involving a two-node network, with each layer of the graph being complete. More specifically, we used the following matrices, obtained with the four-step procedure (note that all numerical values are reported at most to four decimal points):
\begin{align}
A &= 
\begin{bmatrix}
3.2 & 2 \\ 
2 & 3.2 \\ 
\end{bmatrix}
, 
\,\,\,\,\,\,
B =
\begin{bmatrix}
4.2 & 0.312 \\
6.1318 & 2.2 \\ 
\end{bmatrix}
.
\end{align}
The two single virus systems defined using the $A$ and $B$ above (see \eqref{eq:v1} and \eqref{eq:v2}) have the following two endemic equilibria $\bar{x} = [0.8077, 0.8077]^\top$ and $\bar{y} = [0.7801, 0.8699]^\top$,
respectively. These define two survival-of-the-fittest equilibria $(\bar x, \vect 0_n)$ and $(\vect 0_n, \bar y)$ for the bivirus system in \eqref{eq:bivirus_dynamics}. We then obtain $\rho((I-\bar{Y})A) = 0.9276$ and $\rho((I-\bar{X})B) = 0.9436$,
%
which establishes that each survival-of-the-fittest equilibrium is locally exponentially stable, due to Proposition~\ref{prop:boundary_stablility_cond}. For $n = 2$, one can analytically compute coexistence equilibria, see~\citep{ye2022_bivirus}. We used Maple to analytically identify a unique coexistence equilibrium $(\Tilde{x}, \Tilde{y})$:
\begin{align}
    \Tilde{x} 
    = 
    \begin{bmatrix}
    0.5467 &
    0.4180
    \end{bmatrix}^\top
    ,
    \,\,
    \Tilde{y} 
    = 
    \begin{bmatrix}
    0.2418 &
    0.4101
    \end{bmatrix}^\top
    .
\end{align}
The Jacobian matrix of \eqref{eq:bivirus_dynamics} at this coexistence equilibrium has three negative real eigenvalues of $-5.4373$, $-3.8924$ and $-0.7507$ and one unstable eigenvalue of $0.0321$ (see \citep{ye2022_bivirus,ye2022bivirus_survey} for expressions of the Jacobian matrix). 

Fig.~\ref{fig:2a} shows the phase portrait for two initial states in $\intr(\Delta)$ that are close together: Initial state 1, $x(0) = [0.1, 0.1]^\top$ and $y(0) = [0.05, 0.05]^\top$, and Initial state 2, $x(0) = [0.09, 0.09]^\top$ and $y(0) = [0.06, 0.06]^\top$. Different survival-of-the-fittest outcomes occur, with either virus~1 (blue) or virus~2 (red) winning. Figs.~\ref{fig:2b} and \ref{fig:2c} show the time evolution of the blue and red trajectories in Fig.~\ref{fig:2a}, respectively. It is notable that there is a rapid initial transient that takes the system to a point very close to a curve that connects $(\bar x, \vect 0_n)$ to $(\vect 0_n, \bar y)$ and passes through the unstable coexistence equilibrium (magenta cross), followed by a slower convergence to the two survival equilibria.


Separating the time-scales of the two viruses can change the shape of the regions of attraction for $(\bar x, \vect 0_n)$ and $(\vect 0_n, \bar y)$, but the local exponential stability property is unchanged, and thus both regions will always have non-zero Lebesgue measure. Time-scale separation can be easily achieved by introducing a parameter $\gamma > 0$ and modifying \eqref{eq:virus1_dynamics} to be
\begin{equation}\label{eq:virus1_faster}
    \dot x(t) = \gamma\Big(- x(t) + (I-X(t)-Y(t))Ax(t)\Big).
\end{equation}
Adjusting $\gamma$ allows study of scenarios of interest where virus~1 has much faster or slower dynamics relative to virus~2. Fig.~\ref{fig:2d} shows the trajectories for the two viruses having the same speed (green) and virus 1 being faster than virus 2 (purple, $\gamma = 1.2$). We set ``Initial state~1" (square symbol) to be $x(0) = [0.09, 0.09]^\top$ and $y(0) = [0.06, 0.06]^\top$, ``Initial state~2" (circle symbol) as $x(0) = [0.25, 0.25]^\top$ and $y(0) = [0.05, 0.05]^\top$, and ``Initial state~3" ({\Large$*$} symbol) as $x(0) = [0.05, 0.05]^\top$ and $y(0) = [0.2, 0.2]^\top$. 
The unique coexistence equilibrium (magenta cross) obviously remains unchanged for any positive $\gamma$. Thus, from the same initial condition, the virus that survives may depend on the relative speeds of the two viruses, but there are always two nontrivial regions of attraction for the two stable equilibria.

In Fig.~\ref{fig:2e}, we create a $150\times 150$ grid of initial states, imposing a constraint that $x_i(0)+y_i(0) = 0.01$ for $i = 1,2$. For each point on this grid, we simulated the system over a large time window and recorded the particular equilibrium point reached. The figure thus divides the phase plane into two regions of attraction, for initial states that resulted in virus~1 or virus~2 winning the survival-of-the-fittest battle. Fig.~\ref{fig:2f} is generated using the exact same procedure as Fig.~\ref{fig:2e}, except we change the timescale for virus 1 by setting $\gamma=1.2$, resulting in a marked shift in the regions of attraction.It appears that as the dynamics of one virus becomes faster, the region of attraction increases in size, which accords with intuition.


It is known that the region of attraction for an equilibrium forms an open set~\citep{chiang2015stability}, and in our case study there are two locally stable equilibria (the two survival-of-the-fittest equilibria) and two unstable equilibria (the healthy equilibrium and the coexistence equilibrium). Thus, the boundaries of the regions of attraction for $(\bar x, \vect 0_n)$ and $(\vect 0_n, \bar y)$ do not belong to the region of attraction of either, and if the system is initialized at a common point of the two boundaries, then necessarily the trajectories do not converge to either stable equilibrium. In fact, the common boundary of the two regions of attraction forms the stable manifold of the coexistence equilibrium $(\tilde x, \tilde y)$, which has zero Lebesgue measure~\citep{chiang2015stability}. Initial states on this manifold would lead to convergence to $(\tilde x, \tilde y)$, providing a third outcome of the battle, namely coexistence, which is unlikely to be encountered in practice but is not impossible.
Certain trajectories may appear to converge to the unstable equilibrium, but after some time end up moving on to one of the survival-of-the-fittest equilibria.

\subsection{Real-world network case study}
In this section, we present a case study involving a real-world network topology. We consider a mobility network reported in Ref.~\citep{parino2021modelling}, capturing commuting patterns for people between $n=107$ provinces in Italy. The original network $\bar{\mathcal{G}}$, with associated adjacency $\bar A$, is a complete directed graph, i.e., $\bar A$ is a positive matrix but it is not symmetric. Due to differences in commuting patterns, journeys between some provinces were highly frequent, whereas journeys between some other provinces were virtually nonexistent save for a small number of individuals, so the largest and smallest entries of $\bar A$ differed by several orders of magnitude. 
For computational convenience, we first normalize this matrix so that the row sums are equal to $2$, i.e. $\bar A \vect 1_n = 2\vect 1_n$. 
Then, we obtain the matrix $A$ by setting each entry as $a_{ij} =
  \bar a_{ij}$ if $\bar a_{ij} \geq \kappa$, and $a_{ij} = 0$ otherwise,
where the scalar $\kappa > 0$ acts as a threshold. By using $\kappa = 5\times 10^{-5}$, we obtained an $A$ matrix that was nonnegative irreducible but not positive (and thus the graph associated with $A$ is strongly connected but not complete). We normalized $A$ to satisfy $A\vect 1_n = 2\vect 1_n$, and then set $A$ to be the adjacency matrix associated with the network layer of virus~$1$. 

We followed the systematic construction procedure outlined in Section~\ref{ssec:procedure} to obtain $B$. First, note that $\bar x = 0.5\vect 1_n$ due to our normalization. We selected the $48$-th basis vector, i.e., $i = 48$ for the vector $e_i$. The vector $z$ (of dimension $107$) was a vector of zeros, except $z_{48} = 1$ and $z_{55} = -1$, and $\epsilon = 0.2346$.  After following the four step construction procedure, we obtained $B = A$, except the following four entries were adjusted: $b_{48,48}  = a_{48,48}+0.2346$, $b_{48,55}  = a_{48,55} -0.2346$, $b_{48,69}  = a_{48,69} - 0.0205$, $b_{55,100}  = a_{55,100} + 5.0795\times 10^{-4}$.
We can compute that $\rho((I-\bar Y)A) = 0.9999914$ and $\rho((I-\bar X)B) = 0.9999964$, and hence $(\bar x, \vect 0_n)$ and $(\vect 0_n, \bar y)$ are locally exponentially stable. The network is shown in Fig.~\ref{fig:italy_network}, with the matrices $A$ and $B$ found in the online repository mentioned at the start of Section~\ref{sec:simulations}. 

We consider two sets of initial conditions. We first sample $p_i$ and $q_i$ from a uniform distribution $(0,1)$, for all $i \in \{1,\hdots, n\}$. For the first set of initial conditions, we set $x_i(0) = p_i/(p_i+q_i)$ and $y_i(0) = 0.1q_i/(p_i+q_i)$. This ensures that $x_i(0) + y_i(0) < 1$ as required, and the initial virus~$1$ infection level is ten times that of virus~$2$ at any node $i$. For the second set of initial conditions, we set $x_i(0) = p_i/(p_i+q_i)$ and $y_i(0) = 0.5q_i/(p_i+q_i)$. Hence, the initial virus~$1$ infection level is twice as large as that of virus~$2$ for any node $i$. Fig.~\ref{fig:italy_virus1_wins} corresponds to the first set of initial conditions, where virus~$1$ emerges as the winner of the survival-of-the-fittest battle. In contrast, Fig.~\ref{fig:italy_virus2_wins} corresponds to the second set of initial conditions, where virus $2$ wins the battle. Interestingly, virus~$2$ is still able to win the survival-of-the-fittest battle, even if initial infection levels satisfied $x_i(0) = 2 y_i(0)$ for all $i$.


\begin{figure*} 
    \begin{minipage}{0.5\linewidth}
      \subfloat[Virus~1 wins]{\includegraphics[width= 0.8\textwidth]{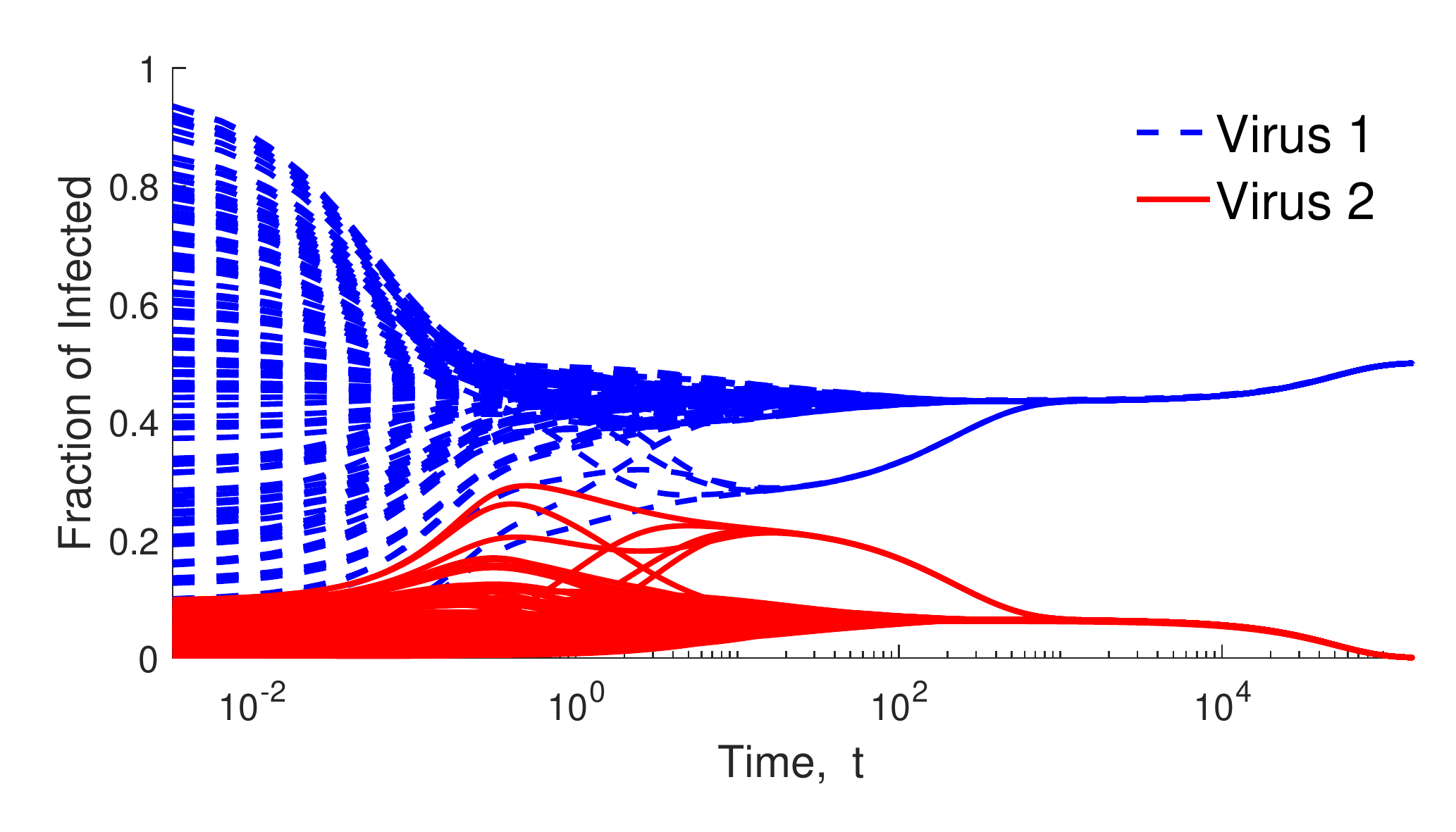}\label{fig:italy_virus1_wins}}
    \vfill
    \subfloat[Virus~2 wins]{\includegraphics[width= 0.8\textwidth]{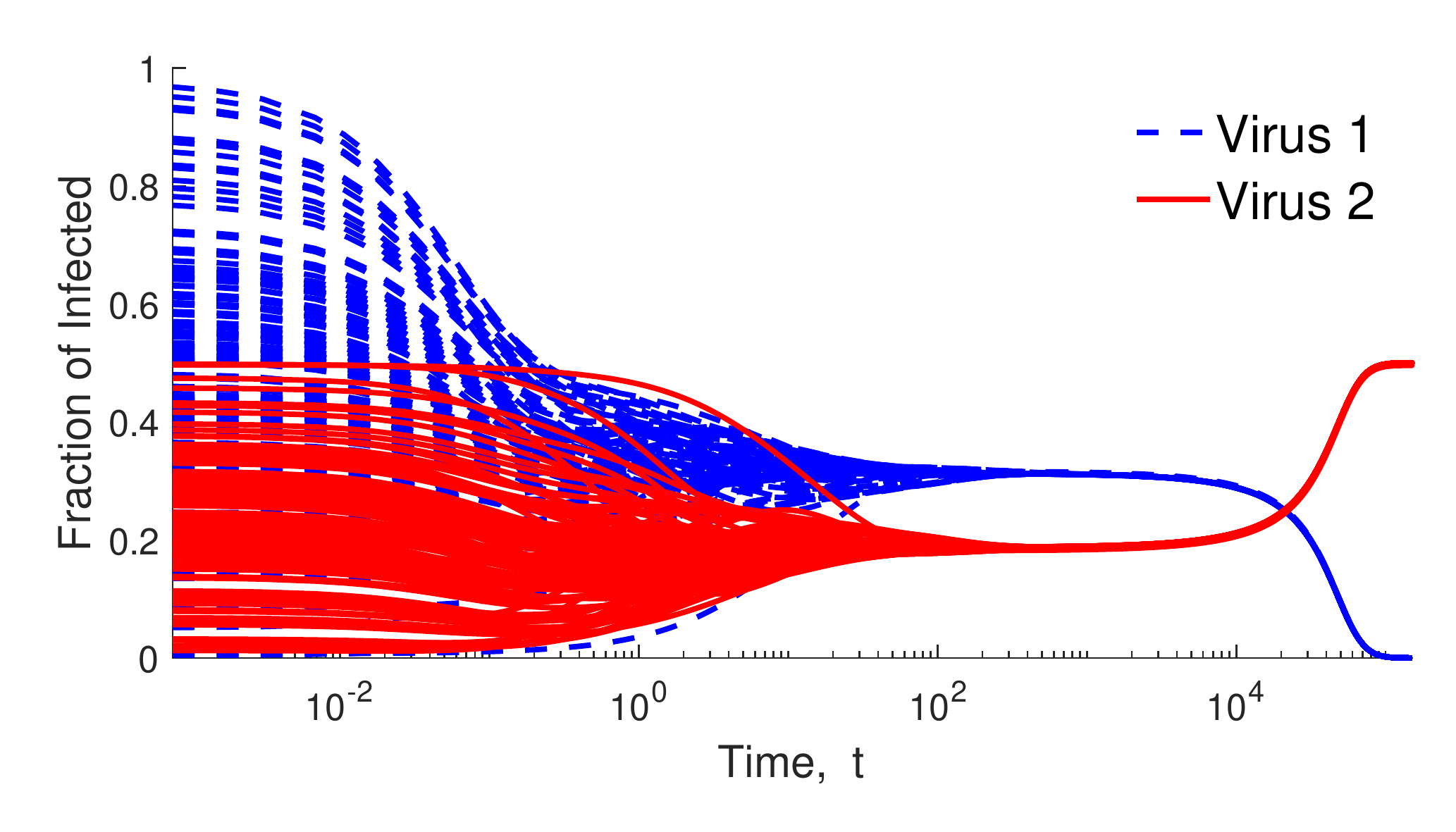}\label{fig:italy_virus2_wins}}
    \end{minipage}
    \hfill
    \begin{minipage}{0.45\linewidth}
      \subfloat[Network of commuting patterns between Italian provinces]{\includegraphics[width= 0.9\textwidth]{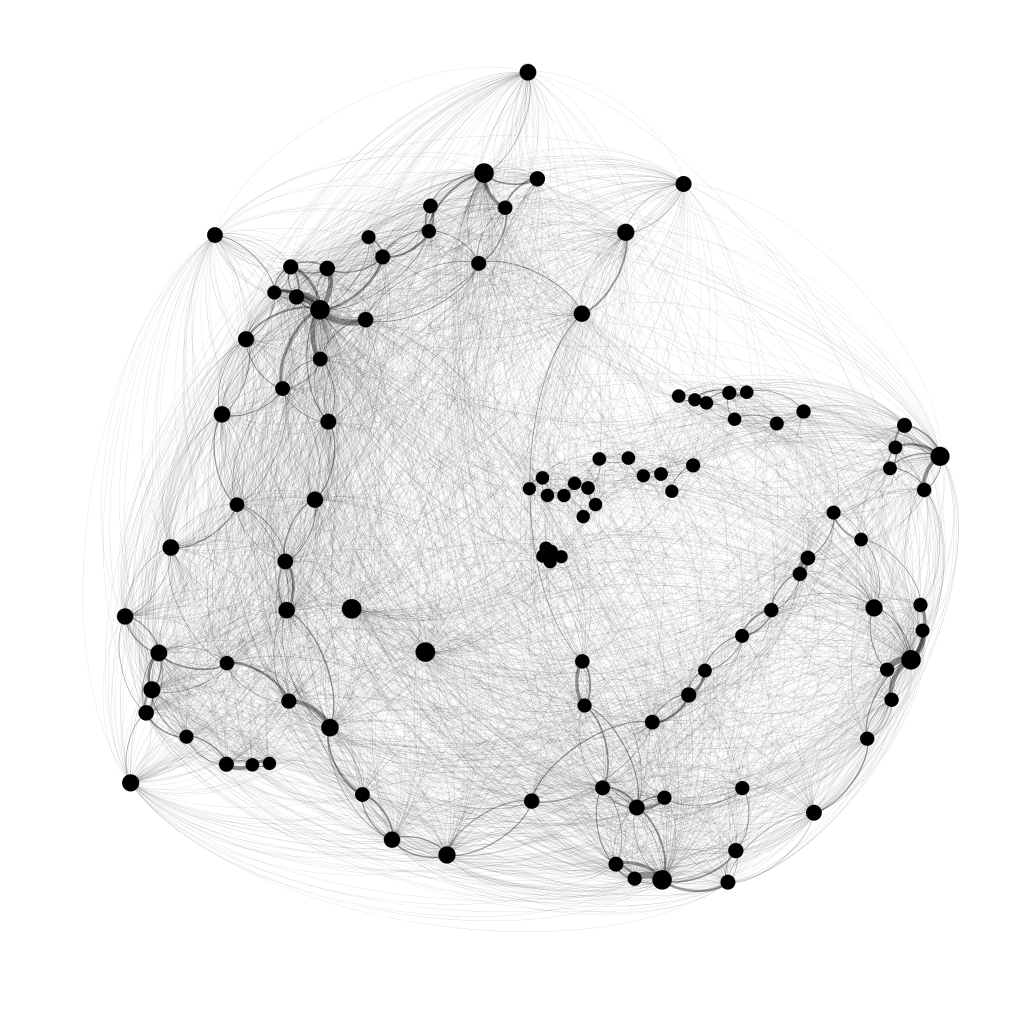}\label{fig:italy_network}}
    \end{minipage}
    \caption{The dynamics of the $n = 107$ example for a mobility network of Italian provinces (note the logarithmic scale of time, $t$, along the horizontal axis). In (a) and (b), the time evolution of $(x(t), y(t))$ shows two different initial states yielding two different survival-of-the-fittest outcomes.}    \label{fig:italy}
\end{figure*}

\section{Conclusion}\label{sec:conclusions}

In summary, we explored a fundamental problem for competing epidemics spreading using the deterministic SIS bivirus network model. 
We provided a rigorous argument which demonstrated that for almost any $A$, there exists a $B$ such that the pair of matrices ($A, B$) satisfied a certain necessary and sufficient condition under which the winner of a survival-of-the-fittest battle depends nontrivially on the initial state of the bivirus network. Then, we presented a systematic procedure to generate such a bivirus network, and studied three numerical examples. This paper significantly expands the known dynamical phenomena of the bivirus model, but should be considered as just a first important step for the epidemic modeling community to explore the diverse new outcomes that are now unlocked for competing epidemic spreading models. A key direction of our future work is to investigate models of three or more competing viruses (multivirus networks)~\citep{janson2020networked}, and to explore how the regions of attraction might change as a function of the relative speeds of the different virus dynamics. Our work did not provide theoretical conclusions on the presence or stability properties of coexistence equilibria, although our example demonstrated a unique unstable coexistence equilibrium; deeper investigation of uniqueness, multiplicity and stability of coexistence equilibria is a current focus~\citep{anderson2022_bivirus_PH}.

\appendix

\section{Proof of Theorem~\ref{thm:doubly_stable_method} and Lemma~\ref{cor:main_paper}}\label{app:main_proof}

In order to prove the claim in Theorem~\ref{thm:doubly_stable_method}, we need the following result on eigenvalue perturbation of a matrix, useful for the subsequent proof.

\begin{lemma}\label{lem:perturb_eig}\cite[pg. 15.2, Fact 3]{hogben2006handbook}. Let $W \in \mathbb{R}^{n\times n}$ be a square matrix with left and right eigenvectors $a^\top, b$ corresponding to a real simple eigenvalue $\lambda$. Suppose that $W$ is perturbed by a small amount $\delta W \in \mathbb{R}^{n\times n}$. Then to first order in $\delta$, $\lambda$ is perturbed by an amount $\delta \lambda$ given by $\delta \lambda=(a^{\top}\delta Wb)/(a^{\top}b)$.
\end{lemma}

\textbf{Proof of Theorem~\ref{thm:doubly_stable_method}}: There are three key steps to the proof. Step 1 deals with the existence claims. Step 2 establishes that the relation \eqref{eq:delta_B} ensures that the bivirus system in \eqref{eq:bivirus_dynamics}, with infection matrices $A$ and $B= B^\prime + \delta B$, has the survival-of-the-fittest equilibrium associated with virus 2 at $({\bf{0}}_n, \bar x+\delta x)$. Step 3 shows that the inequalities \eqref{eq:deltax_1} and \eqref{eq:deltax_2} satisfied by $\delta x$ cause both of the survival-of-the-fittest equilibiria to be locally stable, exploiting inequality conditions in Proposition~\ref{prop:boundary_stablility_cond}. 

\textit{Step 1}. Since $\vect 0_n < \bar x < \vect 1_n$, $I-\bar X$ is a positive diagonal matrix and its inverse exists. Since $u,v$ are assumed to be linearly independent, the row vectors $u^{\top}[\bar X(I-\bar X)^{-1}]$ and $v^{\top}[\bar X(I-\bar X)^{-1}]$ are also linearly independent, and accordingly $\delta x$ exists satisfying \eqref{eq:deltax_1} and \eqref{eq:deltax_2} and its norm can be chosen arbitrarily by scaling. However, an additional requirement has to be met, viz. that \eqref{eq:delta_B} holds for some $\delta B$ such that $B'+\delta B$ is irreducible and nonnegative. This is straightforward if $B'$ is positive by scaling $\delta B$ to have its entries sufficiently small in magnitude, but not when $B'$ can have zero entries. 
To proceed, set $F=(I-X)^{-2}-B'$ and observe that $F=[(I-X)^{-2}-(I-X)^{-1}]+[(I-X)^{-1}-B']$, i.e. $F$ is the sum of a diagonal positive matrix and an irreducible singular $M$-matrix. Hence it is an irreducible nonsingular $M$-matrix, and accordingly $F^{-1}$ is a positive matrix. The two vectors $\tilde u^{\top}:=u^{\top}\bar X(I-\bar X)^{-1}F^{-1}$ and $\tilde v^{\top}:=v^{\top}\bar X(I-\bar X)^{-1}F^{-1}$ are then both positive and linearly independent. Let
\begin{equation}\label{eq:sdef}
s=\delta B\bar x
\end{equation}

Finding $\delta x$ and $\delta B$ to satisfy \eqref{eq:deltax_1}, \eqref{eq:deltax_2} and \eqref{eq:delta_B} is then equivalent to finding $s$ and $\delta B$ to satisfy \eqref{eq:sdef} and
\begin{equation}\label{eq:sineq}
    \tilde u^{\top}s >0\,,\quad \text{and }\quad \tilde v^{\top}s <0.
\end{equation}

We shall now show that these requirements can be fulfilled by choosing just two of the entries of $s$ to be nonzero, and likewise, just two of the entries of $\delta B$. 

Because $\tilde u$ and $\tilde v$ are linearly independent, there is no nonzero $\mu$ for which $\tilde u=\mu \tilde v$. Thus, there must be two integers, say $j$ and $k$, for which $\tilde u_j/\tilde v_j\neq \tilde u_k/\tilde v_k$.
Without loss of generality (using renumbering if necessary) let us assume that $\tilde u_2/\tilde v_2>\tilde u_1/\tilde v_1$, or equivalently, $\tilde u_2/\tilde u_1>\tilde v_2/\tilde v_1$.
Let $\alpha > 0$ be such that
\begin{equation}\label{eq:alpha}
    \alpha\tilde u_2/\tilde u_1>1>\alpha\tilde v_2/\tilde v_1
\end{equation}
Let $\beta>0$ be a constant to be specified below, and choose $s=[-\beta, \alpha\beta, 0, \hdots, 0]^\top$.
Using \eqref{eq:alpha}, it is immediate that \eqref{eq:sineq} is satisfied. Now to specify $\delta B$, observe that the irreducibility of $B'$ guarantees that  at least one entry of the first row and one entry of the second row are positive, say $b'_{1i}$ and $b'_{2j}$. (They may or may not be in the same column.) Choose $\beta$ such that $0< \beta<b'_{1i}\bar x_i $, and set all entries of $\delta B$ to zero except that
\begin{align}\label{eq:delta_B1}
    \delta b_{1i}&=-\beta/\bar x_i\\
    \delta b_{2j}&=\alpha\beta/\bar x_j \label{eq:delta_B2}
\end{align}
These two definitions ensure that $s=\delta B \bar x$ as required, that $B'+\delta B$ is a nonnegative matrix, and that it is also irreducible since it has the same zero entries as $B'$. To summarize, setting $\delta B$ using \eqref{eq:delta_B1} and \eqref{eq:delta_B2} ensures that $B'+\delta B$ is a nonnegative irreducible matrix, while we select the specific form of $s$ with first and second entry equal to $-\beta$ and $\alpha\beta$ to ensure that \eqref{eq:sineq} is satisfied for the particular choice of~$\delta B$.

\textit{Step 2}. We must establish that the survival-of-the-fittest equilibria corresponding to $A$ and $B=B'+\delta B$ are $(\bar x,{\bf{0}}_n)$ and $({\bf{0}}_n, \bar y)$ with $\bar y=\bar x+\delta x+o(\delta)$. Since \eqref{eq:equi_A} holds, the claim for $(\bar x,{\bf{0}}_n)$ is immediate. Next, let $\delta X = \diag(\delta x)$, and observe that
\begin{align}
  &\Big[I-(I-\bar X- \delta X)(B'+\delta B)\Big](\bar x+\delta x)\nonumber \\
  &\qquad =\Big[I-(I-\bar X)B'\Big]\bar x + \delta XB'\bar x-(I-\bar X)(\delta B)\bar x \nonumber \\
  & \qquad \qquad +\Big[I-(I-\bar X)B'\Big]\delta x +o(\delta),
\end{align}
where $o(\delta) = \delta XB'\delta x + \delta X\delta B(\delta x+\bar x)-(I-\bar X)\delta B\delta x$ consists of terms that are of higher order than $\delta$. The first term on the right is zero due to \eqref{eq:equi_B'}. Notice that $B'\bar x=(I-\bar X)^{-1}\bar x$ according to \eqref{eq:equi_B'}, and this is substituted into the second term, where  we also use the fact that $\delta X(I-\bar X)^{-1}\bar x = (I-\bar X)^{-1}\bar X \delta x$. Finally, we use the constraint equation \eqref{eq:delta_B} to handle the third term. Hence, we obtain
\begin{align*}
   \Big[I&-(I-\bar X- \delta X)\Big](B'+\delta B)(\bar x+\delta x) \nonumber \\
  &  =\delta X(I-\bar X)^{-1}\bar x-(I-\bar X)\Big[(I-\bar X)^{-2}-B'\Big]\delta x \nonumber \\ 
  &\qquad \qquad +\Big[I-(I-\bar X)B'\Big]\delta x +o(\delta)\\
  &=\bar X(I-\bar X)^{-1}\delta x-(I-\bar X)^{-1}\delta x+\delta x+o(\delta)\\
  &=o(\delta).
\end{align*}
This establishes the claim for the equilibrium $({\bf{0}}_n,\bar y)$. 

\textit{Step 3.} We must establish that both survival-of-the-fittest eqilibria are locally exponentially stable, i.e. that $\rho[(I-\bar X)B]<1$ and $\rho[(I-\bar Y)A]<1$.  

Consider the effect of a small perturbation $\delta x$ in the entries of $\bar x$ on the Perron--Frobenius eigenvalue of the positive matrix $(I-\bar X)A$; we know the Perron--Frobenius eigenvalue is equal to 1. By Lemma~\ref{lem:perturb_eig} we have (to first order in $\delta$)
\begin{equation}
    \rho[(I-\bar X-\delta X)A]=\rho[(I-\bar X)A]-u^{\top}\delta X A \bar x
\end{equation}
Now use the fact that $[(I-\bar X)^{-1}-A]\bar x={\bf{0}}_n$ to write \mbox{$\rho[(I-\bar X-\delta X)A]=1-u^{\top}\delta X (I-\bar X)^{-1} \bar x$}.
From \eqref{eq:deltax_1}, we have that \mbox{$\rho[(I-\bar X-\delta X)A]<1$} as required. 

The other survival-of-the-fittest equilibrium can be handled similarly. Using arguments like those above, it is evident that $\rho[(I-\bar X)(B'+\delta B)]<1$ if and only if 
\begin{equation}\label{eq:4}
    v^{\top}(I-\bar X)(\delta B)\bar x < 0
\end{equation}
Using the expression for $(\delta B)\bar x$ from \eqref{eq:delta_B}, we have that an equivalent condition to \eqref{eq:4} is $v^{\top}[(I-\bar X)^{-1}-(I-\bar X)B']\delta x<0$.
Recall that $v^{\top}$ is the positive left eigenvector of $I-(I-\bar X)B'$ corresponding to the zero eigenvalue, and so the equivalent condition is
\begin{equation}
    v^{\top}[(I-\bar X)^{-1}-I]\delta x=v^{\top}(I-\bar X)^{-1}\bar X\delta x<0
\end{equation}
This is guaranteed by \eqref{eq:deltax_2}. \hfill \qed
\medskip

\textbf{Proof of Lemma~\ref{cor:main_paper}}: We postpone the proof that $\rho(B')>1$ to the end of the following argument. 

Observe first that because $\bar x > \vect 0_n$, it follows that $z$ must have both positive and negative entries to satisfy $z^\top \bar x = 0$. Next, note that 
\begin{equation}
[I-(I-\bar X) B'] \bar x = [I - (I-\bar X) (A + \epsilon e_i z^\top)] \bar x = \mathbf{0}_n.
\end{equation}
The irreducibility of $A$ implies that there exists a $k\in \mathcal{V}$ such that $a_{ik} > 0$. It follows that for sufficiently small $\epsilon > 0$, $B'=A + \epsilon e_i z^\top$ is nonnegative and irreducible since for any $j\in \mathcal{V}$, $z_j < 0$ only if $a_{ij} > 0$. Therefore  $B'$ fulfills  \eqref{eq:equi_B'}.  We let $u^\top$ and $v^\top$ be positive left eigenvectors of $-I+(I-\bar X)A$ and $-I+(I-\bar X)B^\prime$, respectively, associated with the simple zero eigenvalue (see Lemma~\ref{lem:useful}). Notice that this also implies that $u^\top$ and $v^\top$ are positive left eigenvectors of the nonnegative irreducible matrices $(I-\bar X)A$ and $(I-\bar X)B^\prime$, respectively, associated with the Perron--Frobenius eigenvalue, which is equal to unity. We now need to prove that $u$ and $v$ are linearly independent.

Now, assume, to obtain a contradiction, that $u$ and $v$ are in fact linearly dependent. Then $u^{\top}$ must be a left eigenvector of $(I-\bar X)B'$ with eigenvalue 1. Observe however that this would imply 
\begin{align*}
    u^\top = u^\top (I-\bar X)B' &= u^\top (I-\bar X)(A + \epsilon e_i z^\top) \\
    &= u^\top + u^\top (I-\bar X) \epsilon e_i z^\top \label{eq:3}
\end{align*}
Since $u^\top$ and $\vect 1_n$ are positive vectors, $(I-\bar X)$ is a positive diagonal matrix, and $z^\top \neq \vect 0_n$, it follows that $u^\top (I-\bar X) \epsilon e_i z^\top = u_i(1-\bar x_i)\epsilon z^\top \neq \vect 0_n.$
A contradiction is immediate. Lastly, the fact that $\bar X$ is positive definite with diagonal entries less than $1$ ensures that $\rho(B')>\rho((I-\bar X)B')=1$ (see Item~4 of the properties of $M$-matrices and Metzler matrices in Section~\ref{sec:results}).
\hfill \qed

\bibliographystyle{IEEEtran}
\bibliography{references,MYE_ANU}


\end{document}

%% file: transitions_bivirus_small.pdf_tex
\begingroup%
  \makeatletter%
  \providecommand\color[2][]{%
    \errmessage{(Inkscape) Color is used for the text in Inkscape, but the package 'color.sty' is not loaded}%
    \renewcommand\color[2][]{}%
  }%
  \providecommand\transparent[1]{%
    \errmessage{(Inkscape) Transparency is used (non-zero) for the text in Inkscape, but the package 'transparent.sty' is not loaded}%
    \renewcommand\transparent[1]{}%
  }%
  \providecommand\rotatebox[2]{#2}%
  \newcommand*\fsize{\dimexpr\f@size pt\relax}%
  \newcommand*\lineheight[1]{\fontsize{\fsize}{#1\fsize}\selectfont}%
  \ifx\svgwidth\undefined%
    \setlength{\unitlength}{278.92168903bp}%
    \ifx\svgscale\undefined%
      \relax%
    \else%
      \setlength{\unitlength}{\unitlength * \real{\svgscale}}%
    \fi%
  \else%
    \setlength{\unitlength}{\svgwidth}%
  \fi%
  \global\let\svgwidth\undefined%
  \global\let\svgscale\undefined%
  \makeatother%
  \begin{picture}(1,0.15038598)%
    \lineheight{1}%
    \setlength\tabcolsep{0pt}%
    \put(0.39369077,0.05399987){\color[rgb]{0,0,0}\makebox(0,0)[lt]{\lineheight{1.25}\smash{\begin{tabular}[t]{l}$S$\end{tabular}}}}%
    \put(0.05198875,0.05213325){\color[rgb]{0,0,0}\makebox(0,0)[lt]{\lineheight{1.25}\smash{\begin{tabular}[t]{l}$I$\end{tabular}}}}%
    \put(0.75152636,0.05218792){\color[rgb]{0,0,0}\makebox(0,0)[lt]{\lineheight{1.25}\smash{\begin{tabular}[t]{l}$I$\end{tabular}}}}%
    \put(1.48729651,-0.28792794){\color[rgb]{0,0,0}\makebox(0,0)[lt]{\begin{minipage}{0.96547277\unitlength}\raggedright \end{minipage}}}%
    \put(0,0){\includegraphics[width=\unitlength,page=1]{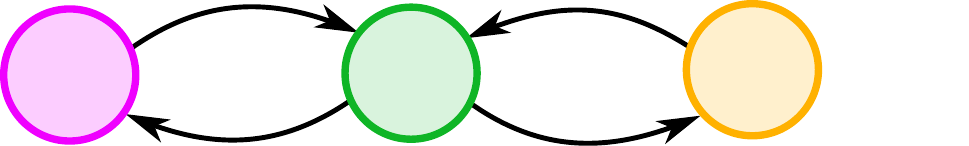}}%
  \end{picture}%
\endgroup%

%% file: two_layer_small.pdf_tex
\begingroup%
  \makeatletter%
  \providecommand\color[2][]{%
    \errmessage{(Inkscape) Color is used for the text in Inkscape, but the package 'color.sty' is not loaded}%
    \renewcommand\color[2][]{}%
  }%
  \providecommand\transparent[1]{%
    \errmessage{(Inkscape) Transparency is used (non-zero) for the text in Inkscape, but the package 'transparent.sty' is not loaded}%
    \renewcommand\transparent[1]{}%
  }%
  \providecommand\rotatebox[2]{#2}%
  \newcommand*\fsize{\dimexpr\f@size pt\relax}%
  \newcommand*\lineheight[1]{\fontsize{\fsize}{#1\fsize}\selectfont}%
  \ifx\svgwidth\undefined%
    \setlength{\unitlength}{440.31430218bp}%
    \ifx\svgscale\undefined%
      \relax%
    \else%
      \setlength{\unitlength}{\unitlength * \real{\svgscale}}%
    \fi%
  \else%
    \setlength{\unitlength}{\svgwidth}%
  \fi%
  \global\let\svgwidth\undefined%
  \global\let\svgscale\undefined%
  \makeatother%
  \begin{picture}(1,0.54303513)%
    \lineheight{1}%
    \setlength\tabcolsep{0pt}%
    \put(0,0){\includegraphics[width=\unitlength,page=1]{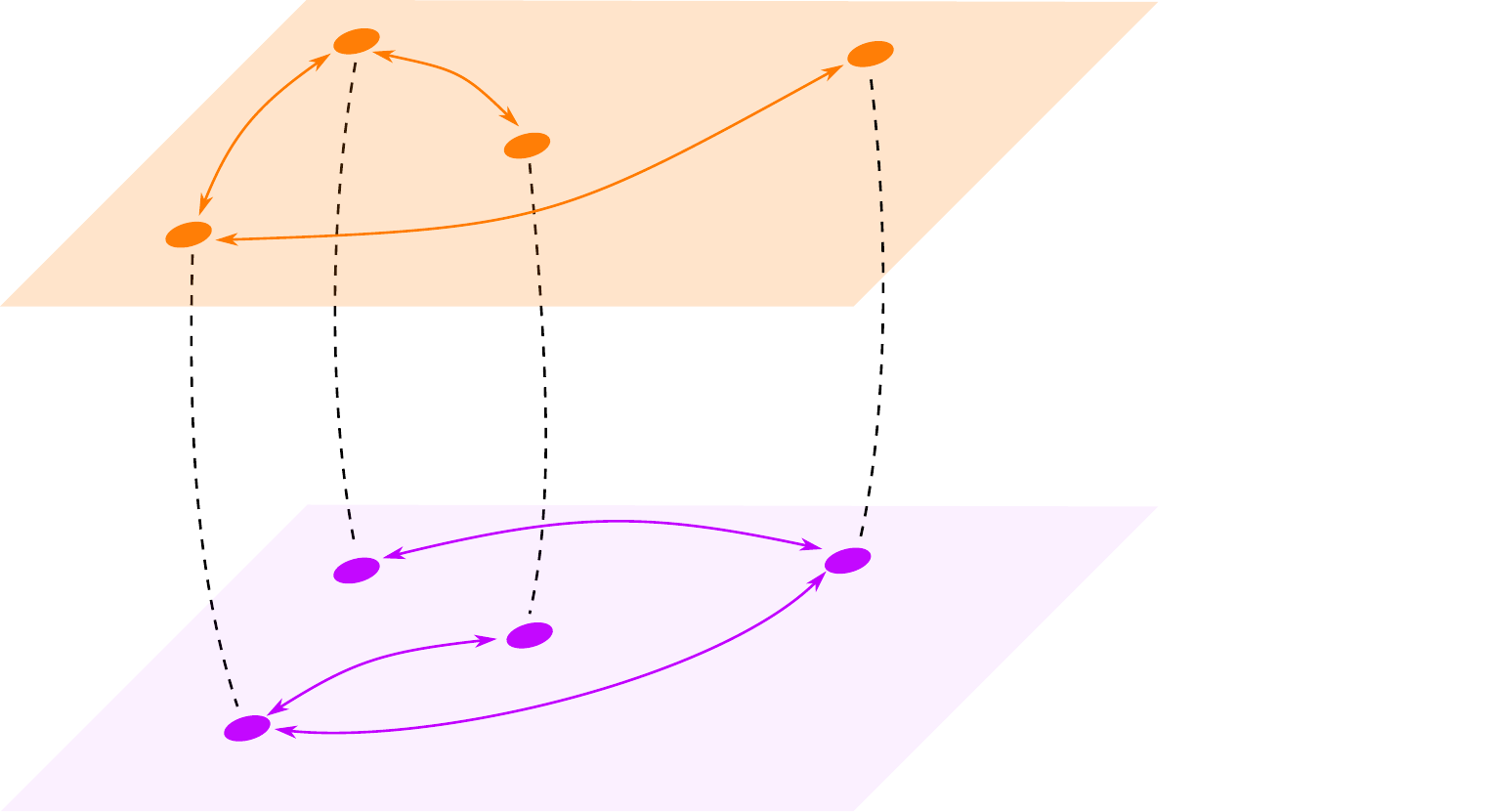}}%
    \put(0.66728338,0.40495862){\color[rgb]{0,0,0}\makebox(0,0)[lt]{\lineheight{1.25}\smash{\begin{tabular}[t]{l} \small Virus 1 layer\end{tabular}}}}%
    \put(0.6461602,0.05307349){\color[rgb]{0,0,0}\makebox(0,0)[lt]{\lineheight{1.25}\smash{\begin{tabular}[t]{l}\small Virus 2 layer\end{tabular}}}}%
  \end{picture}%
\endgroup%

%% file: SCL_Bivirus.bbl
\begin{thebibliography}{10}
\providecommand{\url}[1]{#1}
\csname url@samestyle\endcsname
\providecommand{\newblock}{\relax}
\providecommand{\bibinfo}[2]{#2}
\providecommand{\BIBentrySTDinterwordspacing}{\spaceskip=0pt\relax}
\providecommand{\BIBentryALTinterwordstretchfactor}{4}
\providecommand{\BIBentryALTinterwordspacing}{\spaceskip=\fontdimen2\font plus
\BIBentryALTinterwordstretchfactor\fontdimen3\font minus
  \fontdimen4\font\relax}
\providecommand{\BIBforeignlanguage}[2]{{%
\expandafter\ifx\csname l@#1\endcsname\relax
\typeout{** WARNING: IEEEtran.bst: No hyphenation pattern has been}%
\typeout{** loaded for the language `#1'. Using the pattern for}%
\typeout{** the default language instead.}%
\else
\language=\csname l@#1\endcsname
\fi
#2}}
\providecommand{\BIBdecl}{\relax}
\BIBdecl

\bibitem{hethcote2000mathematics}
H.~W. Hethcote, ``{The Mathematics of Infectious Diseases},'' \emph{SIAM
  Review}, vol.~42, no.~4, pp. 599--653, 2000.

\bibitem{anderson1991_virusbook}
R.~M. Anderson and R.~M. May, \emph{Infectious Diseases of Humans}.\hskip 1em
  plus 0.5em minus 0.4em\relax Oxford University Press, 1991.

\bibitem{wang2019coevolution}
W.~Wang, Q.-H. Liu, J.~Liang, Y.~Hu, and T.~Zhou, ``Coevolution spreading in
  complex networks,'' \emph{Physics Reports}, vol. 820, pp. 1--51, 2019.

\bibitem{newman2013interacting}
M.~E. Newman and C.~R. Ferrario, ``{Interacting Epidemics and Coinfection on
  Contact Networks},'' \emph{PloS One}, vol.~8, no.~8, p. e71321, 2013.

\bibitem{cai2015avalanche}
W.~Cai, L.~Chen, F.~Ghanbarnejad, and P.~Grassberger, ``Avalanche outbreaks
  emerging in cooperative contagions,'' \emph{Nature Physics}, vol.~11, no.~11,
  pp. 936--940, 2015.

\bibitem{gracy2022modeling}
S.~Gracy, P.~E. Par{\'e}, J.~Liu, H.~Sandberg, C.~L. Beck, K.~H. Johansson, and
  T.~Ba{\c{s}}ar, ``Modeling and analysis of a coupled {SIS} bi-virus model,''
  \emph{Automatica, https://arxiv.org/pdf/2207.11414.pdf}, 2022, {N}ote: Under
  Revision.

\bibitem{sahneh2014competitive}
F.~D. Sahneh and C.~Scoglio, ``Competitive epidemic spreading over arbitrary
  multilayer networks,'' \emph{Physical Review E}, vol.~89, no.~6, p. 062817,
  2014.

\bibitem{liu2019bivirus}
J.~Liu, P.~E. Par{\'e}, A.~Nedich, C.~Y. Tang, C.~L. Beck, and T.~Ba\c{s}ar,
  ``{Analysis and Control of a Continuous-Time Bi-Virus Model},'' \emph{IEEE
  Transactions on Automatic Control}, vol.~64, no.~12, pp. 4891--4906, Dec.
  2019.

\bibitem{castillo1989epidemiological}
C.~Castillo-Chavez, H.~W. Hethcote, V.~Andreasen, S.~A. Levin, and W.~M. Liu,
  ``Epidemiological models with age structure, proportionate mixing, and
  cross-immunity,'' \emph{Journal of Mathematical Biology}, vol.~27, no.~3, pp.
  233--258, 1989.

\bibitem{wei2013competing}
X.~Wei, N.~C. Valler, B.~A. Prakash, I.~Neamtiu, M.~Faloutsos, and
  C.~Faloutsos, ``Competing memes propagation on networks: A network science
  perspective,'' \emph{IEEE Journal on Selected Areas in Communications},
  vol.~31, no.~6, pp. 1049--1060, 2013.

\bibitem{watkins2016optimal}
N.~J. Watkins, C.~Nowzari, V.~M. Preciado, and G.~J. Pappas, ``Optimal resource
  allocation for competitive spreading processes on bilayer networks,''
  \emph{IEEE Transactions on Control of Network Systems}, vol.~5, no.~1, pp.
  298--307, 2016.

\bibitem{santos2015bi}
A.~Santos, J.~M. Moura, and J.~M. Xavier, ``Bi-virus sis epidemics over
  networks: Qualitative analysis,'' \emph{IEEE Transactions on Network Science
  and Engineering}, vol.~2, no.~1, pp. 17--29, 2015.

\bibitem{carlos2}
C.~Castillo-Chavez, W.~Huang, and J.~Li, ``Competitive exclusion and
  coexistence of multiple strains in an {SIS STD} model,'' \emph{SIAM Journal
  on Applied Mathematics}, vol.~59, no.~5, pp. 1790--1811, 1999.

\bibitem{yang2017bi}
L.-X. Yang, X.~Yang, and Y.~Y. Tang, ``A bi-virus competing spreading model
  with generic infection rates,'' \emph{IEEE Transactions on Network Science
  and Engineering}, vol.~5, no.~1, pp. 2--13, 2017.

\bibitem{van2014domination}
R.~van~de Bovenkamp, F.~Kuipers, and P.~Van~Mieghem, ``Domination-time dynamics
  in susceptible-infected-susceptible virus competition on networks,''
  \emph{Physical Review E}, vol.~89, no.~4, p. 042818, 2014.

\bibitem{santos2015bivirus_conference}
A.~Santos, J.~M. Moura, and J.~M. Xavier, ``{Sufficient Condition for Survival
  of the Fittest in a Bi-virus Epidemics},'' in \emph{2015 49th Asilomar
  Conference on Signals, Systems and Computers}.\hskip 1em plus 0.5em minus
  0.4em\relax IEEE, 2015, pp. 1323--1327.

\bibitem{pare2021multi}
P.~E. Par{\'e}, J.~Liu, C.~L. Beck, A.~Nedi{\'c}, and T.~Ba{\c{s}}ar,
  ``Multi-competitive viruses over time-varying networks with mutations and
  human awareness,'' \emph{Automatica}, vol. 123, p. 109330, 2021.

\bibitem{janson2020networked}
A.~Janson, S.~Gracy, P.~E. Par{\'e}, H.~Sandberg, and K.~H. Johansson,
  ``Networked multi-virus spread with a shared resource: Analysis and
  mitigation strategies,'' \emph{arXiv preprint arXiv:2011.07569}, 2020.

\bibitem{wang2012dynamics}
Y.~Wang, G.~Xiao, and J.~Liu, ``Dynamics of competing ideas in complex social
  systems,'' \emph{New Journal of Physics}, vol.~14, no.~1, p. 013015, 2012.

\bibitem{ye2022bivirus_survey}
M.~Ye and B.~D. Anderson, ``{Competitive Epidemic Spreading Over Networks},''
  \emph{IEEE Control Systems Letters}, vol.~7, pp. 545--552, 2022.

\bibitem{ackleh2005competitive}
A.~S. Ackleh and L.~J. Allen, ``Competitive exclusion in sis and sir epidemic
  models with total cross immunity and density-dependent host mortality,''
  \emph{Discrete \& Continuous Dynamical Systems-B}, vol.~5, no.~2, p. 175,
  2005.

\bibitem{ye2022_bivirus}
M.~Ye, B.~D.~O. Anderson, and J.~Liu, ``{Convergence and Equilibria Analysis of
  a Networked Bivirus Epidemic Model},'' \emph{SIAM Journal of Control and
  Optimization}, vol.~60, no.~2, pp. S323--S346, 2022, special Section:
  Mathematical Modeling, Analysis, and Control of Epidemics.

\bibitem{berman1979nonnegative_matrices}
A.~Berman and R.~J. Plemmons, \emph{{Nonnegative Matrices in the Mathematical
  Sciences}}, ser. Computer Science and Applied Mathematics.\hskip 1em plus
  0.5em minus 0.4em\relax Academic Press: London, 1979.

\bibitem{santos2014bi}
A.~Santos, ``Bi-virus epidemics over large-scale networks: Emergent dynamics
  and qualitative analysis,'' Ph.D. dissertation, Carnegie Mellon University,
  2014.

\bibitem{fall2007epidemiological}
A.~Fall, A.~Iggidr, G.~Sallet, and J.-J. Tewa, ``Epidemiological models and
  lyapunov functions,'' \emph{Mathematical Modelling of Natural Phenomena},
  vol.~2, no.~1, pp. 62--83, 2007.

\bibitem{Lajmanovich1976}
A.~Lajmanovich and J.~A. Yorke, ``A deterministic model for gonorrhea in a
  nonhomogeneous population,'' \emph{Math. Biosci.}, vol.~28, no.~3, pp.
  221--236, 1976.

\bibitem{anderson2022_bivirus_PH}
\BIBentryALTinterwordspacing
B.~D.~O. Anderson and M.~Ye, ``{Equilibria analysis of a networked bivirus
  epidemic model using Poincar\'e--Hopf and Manifold Theory },'' 2022.
  [Online]. Available: \url{https://arxiv.org/abs/2210.11044}
\BIBentrySTDinterwordspacing

\bibitem{doshi2022convergence}
V.~Doshi, S.~Mallick \emph{et~al.}, ``{Convergence of Bi-Virus Epidemic Models
  With Non-Linear Rates on Networks—A Monotone Dynamical Systems Approach},''
  \emph{IEEE/ACM Transactions on Networking}, 2022.

\bibitem{mei2017epidemics_review}
W.~Mei, S.~Mohagheghi, S.~Zampieri, and F.~Bullo, ``On the dynamics of
  deterministic epidemic propagation over networks,'' \emph{Annual Reviews in
  Control}, vol.~44, pp. 116--128, 2017.

\bibitem{vanMeighem2009_virus}
P.~V. Mieghem, J.~Omic, and R.~Kooij, ``Virus spread in networks,''
  \emph{IEEE/ACM Transactions on Networking}, vol.~17, no.~1, pp. 1--14, 2009.

\bibitem{pare2018analysis}
P.~E. Par{\'e}, J.~Liu, C.~L. Beck, B.~E. Kirwan, and T.~Ba{\c{s}}ar,
  ``Analysis, {E}stimation, and {V}alidation of {D}iscrete-time {E}pidemic
  {P}rocesses,'' \emph{IEEE Trans. on Control Systems Technology}, vol.~28,
  no.~1, pp. 79--93, 2020.

\bibitem{horn1994topics_matrix}
R.~A. Horn and C.~R. Johnson, \emph{Topics in matrix analysis}.\hskip 1em plus
  0.5em minus 0.4em\relax Cambridge University Press, 1994.

\bibitem{qu2009cooperative_book}
Z.~Qu, \emph{{Cooperative Control of Dynamical Systems: Applications to
  Autonomous Vehicles}}.\hskip 1em plus 0.5em minus 0.4em\relax Springer
  Science \& Business Media, 2009.

\bibitem{ye2021_PH_TAC}
M.~Ye, J.~Liu, B.~D.~O. Anderson, and M.~Cao, ``{Applications of the
  Poincar\'e--Hopf Theorem: Epidemic Models and Lotka--Volterra Systems},''
  \emph{IEEE Transactions on Automatic Control}, vol.~67, no.~4, pp.
  1609--1624, Apr. 2022.

\bibitem{chiang2015stability}
H.-D. Chiang and L.~F.~C. Alberto, \emph{{Stability Regions of Nonlinear
  Dynamical Systems: Theory, Estimation, and Applications}}.\hskip 1em plus
  0.5em minus 0.4em\relax Cambridge University Press, 2015.

\bibitem{parino2021modelling}
F.~Parino, L.~Zino, M.~Porfiri, and A.~Rizzo, ``Modelling and predicting the
  effect of social distancing and travel restrictions on {COVID-19}
  spreading,'' \emph{Journal of the Royal Society Interface}, vol.~18, no. 175,
  p. 20200875, 2021.

\bibitem{hogben2006handbook}
L.~Hogben, \emph{Handbook of linear algebra}.\hskip 1em plus 0.5em minus
  0.4em\relax CRC press, 2006.

\end{thebibliography}
